\documentclass[11pt]{article}

\usepackage{amsmath,amssymb,amsthm}
\usepackage{enumerate,color,bbm,ifthen,graphicx,overpic}
\usepackage{mathrsfs}
\usepackage{a4wide}
\usepackage[latin1]{inputenc}
\usepackage{psfrag}
\usepackage{graphics}
\usepackage{subfigure}

\newtheorem{theo}{Theorem}[section]
\newtheorem{lemma}[theo]{Lemma}
\newtheorem{coro}[theo]{Corollary}
\newtheorem{prop}[theo]{Proposition}

\theoremstyle{remark}

\numberwithin{equation}{section}

\def\N{\mathbb{N}}
\def\t{\theta}
\def\ut{\underline{\theta}}
\def\eqsp{\;}
\def\Xset{\mathsf{X}} 
\def\Xsigma{\mathcal{X}} 
\def\Rset{\mathbb R}
\def\tv{\mathrm{TV}} 
\def\hatH{\widehat{H}}
\def\PP{\mathbb{P}} 
\def\PE{\mathbb{E}} \def\F{\mathcal{F}}
\def\g{\gamma_{T_k}}

\newcommand{\eqdef}{\ensuremath{\stackrel{\mathrm{def}}{=}}}
\newcommand{\aslim}{\ensuremath{\stackrel{\mathrm{a.s.}}{\longrightarrow}}}

\newcommand{\un}{\ensuremath{\mathbbm{1}}}
\newcommand{\pscal}[2]{\langle #1, #2 \rangle}
\DeclareMathAlphabet{\mathpzc}{OT1}{pzc}{m}{it}
\newcommand{\dps}{\displaystyle }

\newcommand{\eps}{\varepsilon}
\renewcommand{\leq}{\leqslant}
\renewcommand{\geq}{\geqslant}

\newcounter{hypoconbis}
\newcounter{saveconbis}
\newcommand\debutA{\begin{list} {\textbf{A\arabic{hypoconbis}}}{\usecounter{hypoconbis}}\setcounter{hypoconbis}{\value{saveconbis}}}
\newcommand\finA{\end{list}\setcounter{saveconbis}{\value{hypoconbis}}}
   {\ \\ {\bf Proof #1. }}%
   {\hfill\mbox{\rule{2 true mm}{3 true mm}}\vskip 2 ex\noindent}

\begin{document}

\title{Convergence of the Wang-Landau algorithm}
\author{ G. Fort$^1$, B. Jourdain$^2$, E. Kuhn$^3$, T Leli\`evre$^{2,4}$ and G. Stoltz$^{2,4}$ \\
\footnotesize{1: LTCI, CNRS \& Telecom ParisTech, 46 rue Barrault, 75634 Paris Cedex 13, France} \\
\footnotesize{2: Universit\'e Paris Est, CERMICS, Ecole des Ponts ParisTech, 
6 \& 8 Av. Pascal, 77455 Marne-la-Vall\'ee Cedex 2, France}\\
\footnotesize{3: INRA Unité MIA, Domaine de Vilvert, 78352 Jouy-en-Josas Cedex, France}\\
\footnotesize{4: Project-team MICMAC, INRIA Rocquencourt, Domaine de Voluceau, B.P. 105, 78153 Le Chesnay Cedex, France}  
}

\maketitle

\begin{abstract}
We analyze the convergence properties of the Wang-Landau algorithm. This sampling
method belongs to the general class of adaptive importance sampling
strategies which use the free energy along a chosen reaction coordinate as a
bias.  Such algorithms are very helpful to enhance the sampling properties of Markov Chain
Monte Carlo algorithms, when the dynamics is metastable. We prove the convergence of the Wang-Landau algorithm and  an
associated central limit theorem.
\end{abstract}



\section{Introduction}

The Wang-Landau algorithm was originally proposed in the physics literature to
efficiently sample the density of states of Ising-type
systems~\cite{wang-landau-01,WL01PRL}. From a computational statistical
point of view, it can be seen as some adaptive importance sampling strategy
combined with a Metropolis algorithm: the instrumental distribution is updated
at each iteration of the algorithm in order to have a sampling of the
configuration space as uniform as possible along a given direction.  There are numerous physical and biochemical
works using this technique to overcome sampling problems such as the ones
encountered in the computation of macroscopic properties around critical points
and phase transitions. The original
paper~\cite{WL01PRL} is cited more than one thousand times, according
to Web of Knowledge. The success of the technique motivated its use and study in the statistics
literature, see~\cite{Liang05,LLC07,atchade-liu-10,JR11,BJDD11} for instance
for previous mathematical and numerical studies.

\subsection{Free energy biasing techniques}

\bigskip 

The Wang-Landau algorithm belongs to the class of {\em free energy biasing
techniques}~\cite{lelievre-rousset-stoltz-07-b} which have been introduced in
computational statistical physics to efficiently sample thermodynamic
ensembles and to compute free energy differences. These algorithms can
be seen as {\em adaptive importance sampling techniques}, the biasing factor
being adapted on-the-fly in order to flatten the target probability
measure along a given direction. Let us explain this with more details.

Let $\pi$ be a multimodal probability measure over a high-dimensional space
$\Xset \subseteq \Rset^D$. Classical algorithms to sample $\pi$ (such as a
Metropolis-Hastings procedure with local proposal moves) typically converge
very slowly to equilibrium since high probability regions are separated by low
probability regions. Averages have to be taken over very long trajectories in
order to visit all the modes of the target probability measure $\pi$. The idea
of free energy biasing techniques is to {\em flatten the target probability
  along a well-chosen direction} through an importance sampling procedure in
order to more easily sample $\pi$. More precisely, assume that we are given a
measurable function ${\mathcal O}$ defined on $\Xset$ and with values in a low
dimensional compact space, or in a discrete space. This function is sometimes
called a reaction coordinate or an order parameter in the physics literature.
Let us introduce ${\mathcal O} * \pi$ the image of the measure $\pi$ by
${\mathcal O}$: for any test function $\varphi$ on the image ${\mathcal O}(\Xset)$ of $\Xset$ by ${\mathcal O}$, 
$\int_{{\mathcal O}(\Xset)}\varphi(y){\mathcal O} * \pi(dy)=\int_\Xset\varphi({\mathcal O}(x))\pi(dx)$.
The free energy biased probability measure $\pi^\star$ is defined by the two
following properties: \textit{(i)} the image ${\mathcal O} * \pi^\star$ of
$\pi^\star$ by ${\mathcal O}$ is the uniform measure on ${\mathcal O}(\Xset)$ and \textit{(ii)} for each $y\in{\mathcal O}(\Xset)$, the conditional distributions of $x$ given ${\mathcal O}(x)=y$ under $\pi(dx)$ and $\pi^\star(dx)$ coincide i.e. there exists a measurable function $h:{\mathcal O}(\Xset)\to\Rset_+$ such that 
$\pi^\star(dx)=h({\mathcal O}(x))\pi(dx)$.

Let us give two prototypical examples. When ${\mathcal O}=\xi$ is a smooth function
with values in a continuous space, for example $\xi: \Xset \to {\mathbb T}$
(where ${\mathbb T}={\mathbb R} / {\mathbb Z}$ is the one-dimensional torus), we have
\begin{equation}\label{eq:Ocont}
\pi^\star (dx) = (1/ \rho) \circ \xi (x) \pi (dx) \eqsp,
\end{equation}
where the measure $\xi *
\pi$ is assumed to admit the density $\rho: \mathbb{T} \to \Rset_+$ with respect to the Lebesgue measure on ${\mathbb{T}}$. In this case,
$A(z)=- \ln \rho(z)$ can be interpreted as a free energy~\cite{LRS10}.
This explains the name ``free
energy biasing techniques''. When ${\mathcal O}=I$ is a function with values in a
discrete finite set (this will be the case considered in this paper), $I: \Xset \to \{1, \ldots, d\}$, we have
\begin{equation}\label{eq:Odisc}
\pi^\star (dx) = \frac{1}{d}\sum_{i=1}^d \frac{\un_{I(x)=i}}{\theta_{\star}(i)}
\pi(dx) \eqsp,
\end{equation}
where $\theta_{\star}(i)=\pi\left(\{x\in\Xset, \, I(x)=i\}\right)=I * \pi(i)$ for $i\in\{1, \ldots, d\}$.

The bottom line of free energy biasing techniques is that it should be easier
to sample $\pi^\star$ than to sample $\pi$ since, by construction, ${\mathcal
  O}*\pi^\star$ is the uniform probability measure. Then, sampling from $\pi$
could be obtained by importance sampling from $\pi^\star$. The fact that
$\pi^\star$ is indeed much easier to sample than $\pi$ actually depends on the
choice of ${\mathcal O}$. It is not an easy task to define and to design in
practice a good choice for ${\mathcal O}$ and we do not discuss further these
aspects here. This is related to the choice of a ``good'' reaction coordinate
in the physics literature, which is a very debatable subject. We refer for
example to~\cite{CLS12} for such an analysis in the context of free energy
biasing techniques used to sample posterior distributions in Bayesian
statistics.

Of course, the difficulty is that in general, ${\mathcal O}*\pi$ is unknown
(equivalently $\rho$ in~\eqref{eq:Ocont}, or $\theta_{\star}$
in~\eqref{eq:Odisc}, are unknown) so that it is not possible to sample from
$\pi^\star$. The idea is then {\em to approximate ${\mathcal O}*\pi$ on the
  fly} in order to, in the longtime limit, sample from $\pi^\star$. This is the
adaptive feature of these algorithms: the importance sampling factor is
computed as time goes, in order to penalize states (namely level sets of
${\mathcal O}$) which have already been visited. To approximate $\pi^\star$ at
a given time, one could either use the occupation measure of the Markov chain
up to the current time (this is typically what is done in practice in the
molecular dynamics community) or one could use an approximation over many
Markov chains running in
parallel~\cite{lelievre-rousset-stoltz-07-b,minoukadeh-chipot-lelievre-10}.
Moreover, one could either think of approximating ${\mathcal O}*\pi$ (these are
the so-called Adaptive Biasing Potential (ABP) techniques) or, in the case ${\mathcal
  O}$ is a continuous order parameter, approximating $A'(z)$ (these are the
so-called Adaptive Biasing Force (ABF) techniques~\cite{DP01,HC04}). This gives rise
to many algorithms in the literature (see for instance the classification and references in~\cite{lelievre-rousset-stoltz-07-b}), which
are more or less efficient and more or less difficult to analyze
mathematically. We refer for example to~\cite{DP01,HC04} for ABF techniques
using the occupation measure, to~\cite{minoukadeh-chipot-lelievre-10,JLR10} for
ABF techniques using many replicas in parallel, to~\cite{wang-landau-01} for
ABP approaches using the occupation measure and
to~\cite{babin-roland-sagui-08,BJDD11} for ABP approaches using many replicas
in parallel. Before discussing the efficiency and the mathematical analysis of
these algorithms, let us emphasize that in many applications, computing the
measure ${\mathcal O}*\pi$ (or equivalently the free energy) is actually the
main goal~\cite{LRS10}.

Roughly speaking, from a practical point of view, most ABP approaches
(like the Wang-Landau algorithm) are more
involved to use since they typically require to introduce a vanishing adaption
mechanism. Indeed, even if one starts with a very good approximation
of ${\mathcal O}*\pi$, and thus with a probability measure very close
to $\pi^\star$, the adaptive mechanism will introduce a non-zero biasing factor
to penalize visited level sets of ${\mathcal O}$, as time goes. One
crucial feature of ABP approaches is thus to penalize less and less (as time goes) the
visited states, so that in the longtime limit, no adaption is
performed anymore. The way this adaption mechanism is performed is made precise below in the Wang-Landau case. We would also like to mention that some ABP
techniques without externally imposed vanishing adaption have been
proposed, like the self-healing umbrella sampling~\cite{marsili-barducci-chelli-procacci-schettino-06,dickson-legoll-lelievre-stoltz-fleurat-lessard-10},
but we do not discuss them here.
ABF approaches do not require such a vanishing adaption mechanism
since the approximation of $A'(z)$ is based on conditional measures
given the value of ${\mathcal O}$, which are not affected by the biasing factor
(since it only depends on ${\mathcal O}$). However, ABF techniques
cannot be used for discrete order parameters.

In terms of mathematical analysis, approximations based on many
replicas in parallel are typically easier to analyze, since they can
be related (in the limit of infinitely many replicas) to mean field
models for which powerful longtime convergence analysis techniques can
be used. We refer for example to~\cite{LRS08,lelievre-minoukadeh-09} for such an
analysis for an ABF technique. In~\cite{LRS08} for example, it is shown that the method is
efficient if for each $y\in{\mathcal O}(\Xset)$, the conditional distribution of $x$ given ${\mathcal O}(x)=y$ under $\pi(dx)$ has good mixing properties (namely large Logarithmic
Sobolev Inequality constants). The convergence analysis and, more
importantly, the study of the efficiency of free energy biased
techniques for approximations based on the occupation measure are much
more involved since correlations in time of the Markov process play a
crucial role. The aim of this paper is to propose a convergence analysis for
the Wang-Landau algorithm, which is an ABP approach. 

We would like to stress that the convergence results are a necessary first
step in the study of the Wang-Landau algorithm, but are by no means the end of
the story. Indeed, the real practical interest of adaptive techniques are their
improved convergence properties.  Although this improvement is obvious to
practitioners, it is mathematically more difficult to formalize. We refer to the
companion paper~\cite{FJKLS2} for a study of the efficiency of the
algorithm on a very simple toy model. Further works to mathematically
formalize the efficiency of the Wang-Landau algorithm are required.

\subsection{Objectives and main results}

In this paper, we consider the Wang-Landau algorithm applied to the case ${\mathcal O}=I$ with $I:\Xset\to\{1,\ldots,d\}$. It both computes  a (penalty) sequence $\{\theta_n,
n\geq 0 \}$ approximating (in the longtime limit) the probability measure $I * \pi$ and samples draws $\{X_n,
n\geq 0 \}$ distributed (in the longtime limit) according to $\pi^\star$. The update of the penalty
sequence follows a Stochastic Approximation
algorithm~\cite{RM51,benveniste:metivier:priouret:1987} and  is of the form
\[
\theta_{n+1} = \theta_n + \gamma_{n+1} \mathcal{H}_n(X_{n+1}, \theta_n) \eqsp.
\]
Different strategies about the field $\mathcal{H}_n$ and the adaption schedule
$\{\gamma_n, n\geq 1\}$ have been proposed in the literature.  In the
original paper~\cite{wang-landau-01}, the authors came
up with a stochastic adaption schedule hereafter called {\em flat histogram
  Wang-Landau}. In this procedure, the updating parameter~$\gamma_n$ remains
 constant up to the (random) time when the sampling along the chosen order parameter~${\mathcal O}$ is
 approximately uniform, the ``amount of uniformity'' being measured
 according to the current value of $\gamma_n$. Then $\gamma_n$ is lowered 
 and a new updating procedure of the weights starts with a constant stepsize. Another strategy consists in a deterministic update of
the adaption sequence $\{\gamma_n ,n\geq 1 \}$. 

Despite the Wang-Landau algorithm has been successfully applied for
many problems of practical interest, there are many
open questions about its longtime behavior and its efficiency. Such a longtime
behavior study relies on the convergence of stochastic approximation algorithms
with Markovian
inputs~\cite{benveniste:metivier:priouret:1987,andrieu:moulines:priouret:2005}
combined with the convergence of adaptive Markov chain Monte Carlo
samplers~\cite{fort:moulines:priouret:2011}; for both parts, the stability of
the sequence $\{\theta_n, n\geq 0\}$ is a fundamental
property. Stability here means that the sequence $\{\theta_n, n\geq
0\}$ remains in a compact subset of the probability
measures on $\{1,\ldots,d\}$ with support equal to the support of $I * \pi$ (as
explained in Section~\ref{sec:cvg}, this is related to a recurrence property).

The asymptotic behavior of the flat histogram Wang-Landau algorithm, when
$\mathcal{H}_n$ is such that in some sense, $\theta_n$ counts the number of
visits to the level sets of $\mathcal{O}$, has been considered in~\cite{atchade-liu-10,JR11}. One crucial step is to show that the time $\tau$
to reach the flat histogram criterium is finite with probability one.
In~\cite{atchade-liu-10}, it is proved for a specific field $\mathcal{H}_n$,
that $\tau$ is finite almost-surely, the sequence $\{\theta_n, n\geq 0\}$ is
stable and converges almost-surely. A strong law of large numbers for the draws
$\{X_n,n \geq 0\}$ is also established for a wide family of unbounded
functions. In \cite{JR11}, the authors show that the precise form of $\mathcal{H}_n$
plays a role on the convergence of the flat histogram Wang-Landau
algorithm (see Section~\ref{sec:linearized_WL} for more details).


In this paper, we consider the Wang-Landau algorithm with a deterministic
adaption sequence $\{\gamma_n, n\geq 1\}$ (see again
Section~\ref{sec:linearized_WL} for a precise definition of the algorithm). The aim of this article is to address both the convergence of $\{\theta_n, n\geq 0\}$ to $I * \pi$ and the
convergence of $\{X_n, n\geq 0\}$ to $\pi^\star$.  More precisely, we
prove first that the sequence $\{\theta_n, n\geq 0\}$ is stable, which is a
crucial point for applications: no {\em ad hoc} stabilization
techniques (such as truncation at randomly varying bounds~\cite{chen:guo:gao:1988}) is required.  We also prove the
almost-sure convergence of $\{\theta_n, n\geq 0\}$ as  well as a Central Limit Theorem.  We then prove
the ergodicity and a strong law of large numbers for the draws $\{X_n, n\geq
0\}$.

Concerning the convergence result, we would like to mention the
previous work~\cite{LLC07} where some results about the longtime analysis for Wang-Landau with
deterministic adaption can be found. In this paper, the authors
combine the Wang-Landau algorithm with a reprojection technique on a fixed compact subset of probability
measures on $\{1,\ldots,d\}$ with support equal to the support of $I * \pi$ so that the
sequence $\{\theta_n, n\geq 0\}$ is stable by definition; then, they prove the
convergence of the sequence whenever the limiting point is in the interior of
the reprojection subset.  Therefore, our results extend the work in~\cite{LLC07}
by precisely analyzing the stability of the algorithm, by addressing the
convergence of $\{\theta_n, n\geq 0\}$ under weaker assumptions and by proving
additional asymptotic analysis.

\medskip

The paper is organized as follows. We describe in
Section~\ref{sec:description_algorithm} the algorithm we consider and compare
it to previously proposed Wang-Landau type algorithms. We then study
its asymptotic behavior in Section~\ref{sec:cvg}. We first prove in
Section~\ref{sec:stability} a fundamental stability result. Then we deduce
convergence properties relying on previous results on stochastic approximation with
Markovian inputs and on the theory of adaptive Markov chain Monte Carlo
samplers. The proofs of the results presented in Sections~\ref{sec:cvg}
are gathered in Section~\ref{sec:proof:cvg}.

\section{Description of the Wang-Landau algorithm}
\label{sec:description_algorithm}

\subsection{Notation and preliminaries}

The system that we consider is described by a normalized target probability density
$\pi$ defined on a Polish space $\Xset$, endowed with a reference
measure~$\lambda$ defined on the Borel $\sigma$-algebra $\Xsigma$. Notice that, as for classical Metropolis-Hastings procedure, the practical implementation of the
algorithm only requires to specify $\pi$ up to a multiplicative
constant.
 In statistical physics, $\Xset$ typically is the set of all
admissible configurations of the system while $\pi$ is a Gibbs measure with
density $\pi(x) = Z_\beta^{-1} \exp(-\beta U(x))$, $U$ being the potential
energy function and $\beta$ the inverse temperature.  In condensed matter
physics for instance, actual simulations are performed on systems composed of
$N$ particles in dimension 2 or 3, living in a cubic box with periodic
boundary conditions. In this case, $\Xset = (L\mathbb{T})^{2N}$ or
$\Xset = (L\mathbb{T})^{3N}$, where $L$ is the length of the sides of
the box and
$\mathbb{T} = \mathbb{R} / \mathbb{Z}$ is the one-dimensional torus.

Consider now a partition $\Xset_1, \ldots,
\Xset_d$ of $\Xset$ in $d \geq 2$ elements, and define, for any $i \in \{1, \ldots, d \}$,
\begin{equation}
  \label{eq:def:thetastar}
 \t_\star(i) \eqdef \int_{\Xset_i} \pi(x) \lambda(dx) \eqsp.
\end{equation}
In the following, $\Xset_i$ will be called the $i$-th {\em stratum}.  Each
weight $\t_\star(i)$, which is assumed to be positive, gives the relative
likelihood of the stratum $\Xset_i \subset \Xset$.  In practice, the
partitioning could be obtained by considering some smooth function $\xi \, : \,
\Xset \to [a,b]$ (called a reaction coordinate in the physics literature) and
defining, for $i=1,\dots,d-1$,
\begin{equation}
  \label{eq:def_Xi}
\Xset_i = \xi^{-1}\Big([\alpha_{i-1},\alpha_{i})\Big) \eqsp,
\end{equation}
and $\Xset_d =  \xi^{-1}\left([\alpha_{d-1},\alpha_{d}]\right)$, with 
$a = \alpha_0 < \alpha_1 < \dots \alpha_d = b$ (possibly, $a=-\infty$
and/or $b=+\infty$). In the notation of the introduction, the order
parameter is thus the discrete function $I: \Xset \to \{1, \ldots,
d\}$ defined by
\begin{equation}
  \label{eq:definition:fonctionI}
  \forall x \in \Xset, \, I(x)=i \text{ if and only if } x \in \Xset_i \eqsp.
\end{equation}
As mentioned above, the choice of an appropriate function~$I$ is a difficult issue,
and is mostly based on intuition at the time being: practitioners 
identify some slowly evolving degree of freedom
responsible for the metastable behavior of the system
(the fact that trajectories generated by the numerical 
method remain trapped for a long time 
in some region of the phase space, and only occasionally hop to
another region, where they also remain trapped). There are however ways to 
quantify the relevance of the choice of the reaction coordinate, see for instance
the discussion in~\cite{CLS12}.

The above discussion motivates the fact that the weights $\t_\star(i)$
typically span several orders of magnitude, some sets $\Xset_i$ having very
large weights, and other ones being very unlikely under~$\pi$. Besides,
trajectories bridging two very likely states may need to go through unlikely
regions.  To efficiently explore the configuration space, and sample numerous
configurations in all the strata $\Xset_i$, it is therefore a natural idea to resort
to importance sampling strategies and reweight appropriately
each subset~$\Xset_i$. A possible way to do so is the following.  Let $\Theta$
be the subset of (non-degenerate) probability measures on $\{1, \ldots, d \}$ given by
\[
\Theta = \left\{
\t = (\t(1), \ldots, \t(d)) \ \left| \ 0 < \t(i) <1 \ \text{for all $i \in
  \{1, \ldots, d \}$ and} \ \sum_{i=1}^d \t(i) =1 \right. \right\} \eqsp.
\]
For any $\t \in \Theta$, we define the probability density $\pi_\t$ on $(\Xset,
\Xsigma)$ (endowed with the reference measure $\lambda$) as
\begin{equation}
  \label{eq:def:pitheta}
 \pi_\t(x) = \left(\sum_{i=1}^d \frac{\t_\star(i)}{\t(i)}\right)^{-1} \ 
\sum_{i=1}^d \frac{\pi(x)}{\t(i)} \ \un_{\Xset_i}(x) \eqsp. 
\end{equation}
This measure is such that the weight of the set $\Xset_i$ under $\pi_\t$ is
proportional to $\t_\star(i)/\t(i)$. In particular, all the strata $\Xset_i$
have the same weight under~$\pi_{\t_\star}$. Unfortunately, $\t_\star$ is
unknown and sampling under $\pi_{\theta_\star}$ is typically
unfeasible. 

The Wang-Landau algorithm precisely is a
way to overcome these difficulties: at each iteration of the algorithm, a
weight vector $\t_n = (\t_n(1), \ldots, \t_n(d))$ is updated based on the past
behavior of the algorithm and a point is drawn from a Markov kernel $P_{\t_n}$
with invariant density $\pi_{\t_n}$. The intuition for the convergence of
this algorithm is that if $\{\t_n, n\geq 0 \}$ converges to $\t_\star$ then the
draws are asymptotically distributed according to the density $\pi_{\t_\star}$. Conversely, if the draws are
under $\pi_{\t_\star}$, then the update of $\{\t_n, n\geq 0 \}$ is chosen such that it
converges to $\t_\star$. We will derive below sufficient conditions on the sequence $\{\gamma_n,n\geq 1\}$ of step-sizes used to update $\{\t_n, n\geq 0 \}$ and on the Markov kernels $\{P_\t, \t \in
\Theta\}$ in order to prove the convergence of a version of the Wang-Landau
algorithm, namely a linearized Wang-Landau algorithm with a
deterministic adaption where the step-size $\gamma_n$ is used at the $n$-th
iteration of the Markov chain.

\subsection{The linearized Wang-Landau algorithm with  deterministic
  adaption}
\label{sec:linearized_WL}

 We now
describe the algorithm we study in this article.  Let $\{\gamma_n, n \geq 1\}$
be a $[0,1)$-valued deterministic sequence. For any $\t \in \Theta$, denote by
$P_\t$ a Markov transition kernel onto $(\Xset, \Xsigma)$ with unique
stationary distribution $\pi_\t(x)\lambda(dx)$; for example, $P_\t$ is
one step of a Metropolis-Hastings
algorithm~\cite{MRRTT53,Hastings70} with target probability measure $\pi_\t(x)\lambda(dx)$.

Consider an initial value $X_0 \in \Xset$ and an initial set of weights $\t_0
\in \Theta$ (typically, in absence of any prior information, $\t_0(i) = 1/d$).
Define the process $\{(X_n, \t_n), n\geq 0 \}$ as follows: given the current
value $(X_n, \t_n)$,
\begin{enumerate}[\quad (1)]
\item Draw $X_{n+1}$ under the conditional
  distribution $P_{\t_n}(X_n, \cdot)$;
\item Set $i = I(X_{n+1})$ where $I$ is given by
  (\ref{eq:definition:fonctionI}). The weights are then updated as
\begin{equation}
\label{eq:update_weights_linearized}
\left\{ \begin{array}{ll}
    \t_{n+1}(i) =   \t_n(i) + \gamma_{n+1} \ \t_n(i)  \left(1 -  \t_n(i)\right),  & \\
    \t_{n+1}(k) = \t_n(k) - \gamma_{n+1} \ \t_n(k) \ \t_n(i) & \text{for} \ k \neq i.
\end{array} \right.
\end{equation}
\end{enumerate}
Note that since $\gamma_n \in [0,1)$, $\t_{n} \in \Theta$ for any $n
\geq 0$. As explained in the introduction, the idea of the updating
strategy~\eqref{eq:update_weights_linearized} is that the weights of the
visited stratas are increased, in order to penalize already visited states. 
The update of the probability vector $\t_n$ can be recast equivalently
into the stochastic
approximation framework upon writing
\begin{equation}
  \label{eq:SSA_formulation}
  \t_{n+1} = \t_n + \gamma_{n+1} \, H(X_{n+1},\t_n) \eqsp,
\end{equation}
where $H: \Xset \times \Theta \to [-1,1]^d$ is defined componentwise by
\begin{equation}
  \label{eq:def:ChampsH}
 H_i(x,\t) = \t(i) \left( \un_{\Xset_i}(x) - \t(I(x)) \right)\eqsp, 
\end{equation}
with the function $I$ given by (\ref{eq:definition:fonctionI}).

The updating strategy~(\ref{eq:update_weights_linearized}) (or
equivalently~\eqref{eq:SSA_formulation}) is a modification of the original
Wang-Landau algorithm obtained by (i) using a deterministic schedule for the evolution of the step-sizes used to modify
the values of the weights (instead of reducing the value of these step-sizes at
random times when the empirical frequencies of the strata are sufficiently uniform: this is the flat histogram version of the
Wang-Landau algorithm mentioned in the introduction) and (ii) linearizing at first order in
$\gamma_n$ the update of the weight $\t_n$.

Concerning this second
point,  the standard Wang-Landau update is 
\begin{equation}
\label{eq:NL_update}
\t_{n+1}(i) = \t_n(i) \frac{1 + \gamma_{n+1} \un_{\Xset_i}(X_{n+1})}{\dps 1 + \gamma_{n+1}\theta_n(I(X_{n+1}))}\eqsp.
\end{equation}
The update~\eqref{eq:update_weights_linearized} is obtained
from~\eqref{eq:NL_update} in the limit of small $\gamma_n$. For the stability
and the convergence analysis in Section~\ref{sec:cvg}, we adopt this linear
update. The main advantage is that it makes the proof of convergence
simpler: with the  standard Wang-Landau update, an additional remainder term would
have to be considered in Proposition~\ref{prop:remainder:cvg}. Nevertheless, since $\gamma_n$ converges to zero, the stability and convergence results
stated in Section~\ref{sec:cvg} could be proved using similar arguments for the standard Wang-Landau update (see Section~\ref{sec:stabnl} below concerning the stability).

By contrast, we would like to emphasize here that this distinction between the two
updating strategies~\eqref{eq:update_weights_linearized}
and~\eqref{eq:NL_update} {\em does matter} when
considering the flat histogram criterium for the vanishing adaption
procedure, as proved in~\cite{JR11}. 
Indeed it is shown in~\cite{JR11} that the linearized version of the update~\eqref{eq:update_weights_linearized} allows to
satisfy in finite time the uniformity criterion required in the original
Wang-Landau algorithm, whereas this is not guaranteed for the nonlinear
update~\eqref{eq:NL_update}.

\section{Convergence of the Wang-Landau algorithm}
\label{sec:cvg}
The proof of the convergence of the Wang-Landau algorithm described in
Section~\ref{sec:linearized_WL} relies on its
reformulation~\eqref{eq:SSA_formulation} as a stochastic approximation
procedure. Since the draws $\{X_n, n\geq 1\}$ satisfy for any measurable
non-negative function $f$:
\begin{equation}
  \label{eq:MarkovDynamic}
  \PE\left[f(X_{n+1}) \vert \F_n \right] = P_{\t_n}f(X_n) \eqsp,
\end{equation}
where $\F_n$ denotes the $\sigma$-field $\sigma(\t_0, X_0, X_1, \ldots, X_n )$,
it is a so-called "stochastic approximation procedure with Markovian dynamics" (see
\textit{e.g.}~\cite{benveniste:metivier:priouret:1987}).

The main difficulty, when proving the
almost-sure convergence of such algorithms, is the stability, namely how to
ensure that the sequence $\{\t_n, n\geq 0\}$ remains in a compact subset of~$\Theta$.  We use a traditional approach to answer this question: we first
prove that our algorithm satisfies a recurrence property {\em i.e.}  the sequence $\{\t_n, n\geq 0\}$
visits infinitely often a compact subset of~$\Theta$; we then show that there
exists a Lyapunov function with respect to the {\em mean-field} function $h:
\Theta \to [-1,1]^d$
\begin{equation} \label{eq:mean-field}
  h(\t) =\int_\Xset H(x,\t) \, \pi_\t(x) \, \lambda(dx) = \left( \sum_{j=1}^d
    \frac{\theta_\star(j)}{\theta(j)}\right)^{-1} ( \theta_\star -
  \theta) \eqsp, 
\end{equation}
with strong enough properties so that the recurrence property implies
stability.  Different strategies based on truncations are proposed in the
literature to circumvent the stability problem (see
\textit{e.g.}~\cite{kushner:yin:1997}). The most popular technique is the
truncation to a fixed compact set but this is not a satisfactory solution since
the choice of this compact is delicate: a necessary condition for convergence
is that the compact contains the unknown desired limit. An adaptive truncation
has been proposed by~\cite{chen:guo:gao:1988} (and analyzed for
example in~\cite{andrieu:moulines:priouret:2005,lelong}) which avoids the main drawbacks
of the deterministic truncation approach.  We prove in
Section~\ref{sec:stability} that, under conditions on the target density $\pi$
and the step-size sequence $\{\gamma_n, n\geq 1\}$, the algorithm
(\ref{eq:SSA_formulation}) is recurrent, so that such truncation techniques are
not required.

In Section~\ref{sec:convergence_weight}, we address the almost-sure
convergence of the weight sequence $\{\t_n, n\geq 0\}$.  We then obtain in
Section~\ref{sec:convergence_averages} the convergence in distribution and a
strong Law of large numbers for the samples $\{X_n, n \geq 0 \}$. Finally, we
obtain a central limit theorem in Section~\ref{sec:asymptotic_variance_weight}
for the weight sequence $\{\t_n, n\geq 0 \}$.

\subsection{Assumptions on the Metropolis dynamics and on the  adaption rate}
\label{sec:assumptions}
Our conditions fall into three categories: conditions on the equilibrium
measure (see~A\ref{hyp:targetpi}), on the transition kernels $\{P_\t, \t
\in\Theta \}$ (see~A\ref{hyp:kernel}) and conditions on the step-size sequence
$\{\gamma_n, n \geq 1\}$ (see~A\ref{hyp:stepsize}). It is assumed that

\debutA
\item \label{hyp:targetpi} The probability density $\pi$ with respect to the
  measure $\lambda$ is such that $0 < \inf_{\Xset}\pi \leq \sup_\Xset \pi <
  \infty$.  In addition, $\inf_{1 \leq i \leq d } \t_\star(i)> 0$ where
  $\t_\star$ is given by~(\ref{eq:def:thetastar}).  \finA

  The first part of Assumption A\ref{hyp:targetpi} is satisfied, for example, for smooth
  positive densities on a compact state space $\Xset \subset \Rset^D$ with the
  Lebesgue measure as the reference measure $\lambda$, or for a positive
  probability measure on a discrete finite state space $\Xset=\{1,\ldots,K\}$
  with the uniform measure as the reference measure. Since $\inf_\Xset \pi$ is
  assumed to be positive, the second part of the assumption is satisfied as
  soon as $\inf_{1\leq i\leq d}\lambda(\Xset_i) > 0$. The minorization
  condition on $\pi$ certainly is the most restrictive assumption: it is
  introduced in order to prove the recurrence of the
  algorithm~(\ref{eq:SSA_formulation}). This condition can be removed by adding
  a stabilization step to (\ref{eq:SSA_formulation}) (such as a truncation
  technique at random varying bounds~\cite{chen:guo:gao:1988,kushner:yin:1997})
  in order to ensure the recurrence.
  
The second assumption is:
\debutA
  \item \label{hyp:kernel} For any $\t \in \Theta$, $P_\t$ is a
    Metropolis-Hastings transition kernel with invariant distribution $\pi_\t \ 
    d\lambda$, where $\pi_\t$ is given by (\ref{eq:def:pitheta}), and with symmetric
    proposal kernel $q(x,y) \lambda(dy)$ satisfying $\inf_{\Xset^2} q >
    0$.
\finA 
The transition probability for a symmetric Metropolis-Hastings dynamics reads
\[
P_\t(x,dy) = q(x,y) \alpha_\t(x,y) \, \lambda(dy) + \delta_x(dy) \ \left(1-\int_\Xset q(x,z) \alpha_\t(x,z) \, \lambda(dz) \right) \eqsp,
\]
with 
\[
\alpha_\t(x,y) = 1 \wedge \frac{\pi_\t(y)q(y,x)}{\pi_\t(x)q(x,y)}=1 \wedge \frac{\pi_\t(y)}{\pi_\t(x)} \eqsp,
\]
the last equality being a consequence of the symmetry of $q$.  Assumption
A\ref{hyp:kernel} is satisfied for instance when $\Xset = \mathbb{T}^n$ (a
cubic simulation cell endowed with periodic boundary conditions), and $q(x,y) =
\widetilde{q}(y-x)$ for a positive and even density $\widetilde{q}$ such that
$\inf_\Xset \widetilde{q} >0$.

The minorization condition on $q$ implies
that the transition kernels $\{ P_\t, \t \in \Theta \}$ are uniformly
(geometrically) ergodic, as stated in Proposition~\ref{prop:unifergo} below.  This property allows a simple presentation of the
main ingredients for the limiting behavior analysis of the algorithm.
Extensions to a more general case could be done by using the same tools as in
\cite{fort:moulines:priouret:2011} (see also~\cite[Section 3]{andrieu:moulines:priouret:2005}) and controlling the dependence upon $\t$ of
the ergodic behavior. These technical steps are out of the scope of
this paper.

We prove in Section~\ref{proof:prop:unifergo} the following result:
\begin{prop}
  \label{prop:unifergo}
  Under A\ref{hyp:targetpi} and A\ref{hyp:kernel}, there exists $\rho
  \in (0,1)$ such that for all $\t \in \Theta$, for all $x \in \Xset$
  and for all $A \in \Xsigma$, it holds:
  \begin{align} 
    & \qquad P_\t(x,A) \geq \rho \, \int_A
    \pi_\t(x) \, \lambda(dx) \eqsp, \label{eq:minorization_bound} \\
   & \sup_{\t \in \Theta} \sup_{x \in \Xset} \left\| P_\t^n(x, \cdot) - \pi_\t \, d\lambda
    \right\|_\tv \leq 2  (1-\rho)^n, \label{eq:uniformeergodicite}
  \end{align}
where for a signed measure $\mu$, the total variation norm
is defined as
\[
\|\mu\|_\tv = \sup_{\{f \, : \, \sup_\Xset |f| \leq 1 \}} |\mu(f)| \eqsp.
\]
\end{prop}

We finally introduce conditions on the magnitude of the step-size sequence.
\debutA
\item \label{hyp:stepsize} The sequence $\{\gamma_n, n \geq 1\}$ is a
  $[0,1)$-valued deterministic sequence such that 
  \begin{enumerate}[a)]
  \item \label{hyp:pasdecroiss} $\{\gamma_n, n \geq 1 \}$ is a non-increasing
    sequence and $\lim_n \gamma_n = 0$;
  \item \label{hyp:stepsize:item2} $\sum_n \gamma_n  =\infty$; 
  \item \label{hyp:stepsize:item3} $\sum_n \gamma_n^2 < \infty$.
   \end{enumerate}
   \finA For ease of exposition, it is assumed in
   A\ref{hyp:stepsize} that the sequence is
   non-increasing and with values strictly smaller than $1$. These
   hypotheses can be weakened by assuming that  they are only
   satisfied ultimately: for
   some constant $n_0$, the sequence $\{\gamma_n, \, n \ge n_0\}$ is
   non-increasing  and with values strictly smaller than $1$, and 
   $\gamma_n \le 1$ for $n < n_0$. Examples of step-size sequence satisfying
   assumption A\ref{hyp:stepsize} are the polynomial schedules $\gamma_n =
   \gamma_\star/n^\alpha$ with $1/2 < \alpha \leq 1$. As already observed in
   Section~\ref{sec:linearized_WL}, the condition $\gamma_n \in [0,1)$ implies
   that if $\t_0 \in \Theta$, then for any $n \geq 1$, $\theta_n \in \Theta$.
   Assumption A\ref{hyp:stepsize}\ref{hyp:pasdecroiss} is introduced for the
   proof of the recurrence property. Assumptions
   A\ref{hyp:stepsize}\ref{hyp:stepsize:item2}-\ref{hyp:stepsize:item3} are
   standard conditions for the stability and the convergence of a stochastic
   approximation scheme since the pioneering work~\cite{RM51}.

\subsection{Recurrence property of the weight sequence  $\{\t_n, n \geq 0\}$}
\label{sec:stability}
We state in this section that, almost surely, there
exists a compact subset of $\Theta$ such that $\theta_n$ belongs to this
compact subset for infinitely many $n$. For any $n \ge 0$, set
\begin{equation}
  \label{eq:def:underlinetheta}
  {\underline{\theta}_n} = \min_{1\leq j\leq d}\theta_n(j) \eqsp.
\end{equation}
We prove in Section~\ref{sec:proof_recomp} the following theorem:
\begin{theo}
  \label{recomp}
  Assume A\ref{hyp:targetpi}, A\ref{hyp:kernel} and
  A\ref{hyp:stepsize}\ref{hyp:pasdecroiss}. Then, $\dps \PP\left( \limsup_{n\to\infty}\underline{\theta}_n>0 \right) =1$.
\end{theo}
The proof is based on the following consideration. The
value of the smallest weight increases when the chain goes into the
corresponding stratum (see the updating
formula~(\ref{eq:update_weights_linearized})). Under the stated assumptions, we prove that the chain $\{X_n, n\geq 0\}$
returns in the strata of smallest weights often enough for the smallest
weight to remain isolated from~$0$.

\subsection{Convergence of the weight sequence  $\{\t_n, n \geq 0\}$}
\label{sec:convergence_weight}
In this subsection, the almost-sure convergence of the sequence $\{\t_n, n\geq
0 \}$ to $\t_\star$ is addressed. We prove in Section~\ref{proof:theo:cvg} the
following  convergence result:
\begin{theo}
  \label{theo:cvg}
  Assume A\ref{hyp:targetpi}, A\ref{hyp:kernel} and A\ref{hyp:stepsize}. Then, $\dps \PP\left( \lim_{n\to\infty} \t_n = \t_\star \right) = 1$.
\end{theo}

The proof relies on~\cite{andrieu:moulines:priouret:2005} which provides
sufficient conditions for convergence of stochastic approximation techniques.
The first step consists in rewriting the weight
update~\eqref{eq:SSA_formulation} as
\begin{equation}
  \label{eq:reformulation_using_mean_field}
  \t_{n+1} = \t_n + \gamma_{n+1} h(\t_n) + \gamma_{n+1} \Big(H(X_{n+1},\t_n) -h(\t_n) \Big) \eqsp,
\end{equation}
where $h$ is given by (\ref{eq:mean-field}).  The heuristic idea is that, if
the step-size quickly is sufficiently small, and the Metropolis dynamics converges sufficiently
fast to equilibrium for $\t$ fixed (a result given by
Proposition~\ref{prop:unifergo}), the update of $\t_n$ is indeed close to an
update with the averaged drift $h(\t_n)$. However, in order for the updates of
the weights to be non-negligible, the step-sizes should not be too small. The
balance between these two opposite effects is encoded in the conditions
A\ref{hyp:stepsize}\ref{hyp:stepsize:item2}-\ref{hyp:stepsize:item3}.

{F}rom a technical viewpoint, the proof of the theorem relies on two main tools.
The first one (see Proposition~\ref{prop:Lyapunov}) is to show that the
function $V:\Theta \to \Rset_+$ given by
\begin{equation}
  \label{eq:def:Lyapunov:V}
  V(\t) \eqdef \sum_{i=1}^d \t_\star(i) \log\left(
  \frac{\t_\star(i)}{\t(i)}\right)
\end{equation}
is a Lyapunov function with respect to the mean-field $h$, namely $\langle
\nabla V(\theta), h(\theta) \rangle < 0$ for $\t \neq \t_\star$ and $\langle
\nabla V(\theta_\star), h(\theta_\star) \rangle = 0$ (here, $\langle
\cdot,\cdot\rangle$ denotes the scalar product in $\Rset^d$).  This motivates
the fact that $\{\t_n, n\geq 0\}$ may converge to $\t_\star$.  The second
important result establishes that the remainder term $\gamma_{n+1}
\left(H(X_{n+1},\t_n) -h(\t_n) \right)$
in~\eqref{eq:reformulation_using_mean_field} vanishes in some sense (see
Proposition~\ref{prop:remainder:cvg}). This step is quite technical and
requires regularity-in-$\t$ of the transition kernels $P_\t$ and the invariant
distributions $\pi_\t$ (see Lemmas~\ref{lem:RegularitePi}
and~\ref{lem:RegulariteKernel}). The conclusion then follows
from~\cite[Theorem~2.3]{andrieu:moulines:priouret:2005} and
Theorem~\ref{recomp}.

\subsection{Ergodicity and Law of large numbers for the samples $\{X_k, k \geq 0\}$}
\label{sec:convergence_averages}

In this subsection, we discuss the asymptotic behavior of the chain $\{X_k, k
\geq 0\}$. The main result is the following (see Section~\ref{proof:theo:ergoLLN:X}
for the proof).

\begin{theo}
\label{theo:ergoLLN:X}
Assume A\ref{hyp:targetpi}, A\ref{hyp:kernel} and A\ref{hyp:stepsize}. Then, for
any bounded measurable function $f$,
  \begin{align}
    & \lim_{n\to\infty} \PE\left[f(X_n)\right] = \int_\Xset f(x) \, \pi_{\t_\star}(x) \, 
    \lambda(dx)
    \eqsp, \label{theo:ergoLLN:X:item1} \\
    & \frac{1}{n} \sum_{k=1}^n f(X_k) \aslim \int_\Xset f(x) \, \pi_{\t_\star}(x) \, \lambda(dx) \eqsp. \label{theo:ergoLLN:X:item2}
  \end{align}
\end{theo}

This theorem shows that the distribution of the sample $X_n$
 converges to $\pi_{\t_\star}(x) \lambda(dx)$, where, we recall
\[\pi_{\t_\star}(x)=
\frac{1}{d} \sum_{i=1}^d \frac{\pi(x)}{\t_\star(i)} \un_{\Xset_i}(x)
\eqsp.
\]
Moreover, the empirical mean of the samples $\{f(X_k), k \geq 0\}$ converges to $\int
f \, \pi_{\t_\star} \, d\lambda$. Hence, although the weights $\t_n$ evolve in the adaptive
algorithm, ergodic averages can be thought of as averages with fixed weights $\t_\star$.

In many practical cases, averages with respect to~$\pi$ are of interest.  In
this case, the Wang-Landau procedure is used as some adaptive importance
sampling strategy. In order to obtain averages according to~$\pi$ along a
trajectory of the algorithm, some reweighting has to be considered. A natural
strategy is to use some stratified-type weighted sum of the samples $\{X_k, \,
k\geq 1\}$:
\[
\mathcal{I}^{}_n(f) \eqdef d \sum_{i=1}^d \t_n(i) \left(  
\frac{1}{n}\dps \sum_{k=1}^nf(X_k) \un_{\Xset_i}(X_k)\right).
\] 
We prove in Section~\ref{proof:theo:ergoLLN:Stratified} the following result:
\begin{theo}
  \label{theo:ergoLLN:Stratified}
Assume A\ref{hyp:targetpi}, A\ref{hyp:kernel} and A\ref{hyp:stepsize}. Then for
any bounded measurable function $f$,
\begin{align}
  \label{theo:ergoLLN:Stratified:item1}
  & \lim_{n\to\infty} d \, \PE\left[ \sum_{i=1}^d \t_n(i) \, f(X_n) \, 
    \un_{\Xset_i}(X_n) \right] = \int_\Xset f(x) \, \pi(x) \, \lambda(dx) \eqsp, \\
  \label{theo:ergoLLN:Stratified:item2}
 & \mathcal{I}_n(f) \aslim \int_\Xset f(x) \, \pi(x) \, \lambda(dx) \eqsp.
  \end{align}
\end{theo}
There are of course many other reweighting strategies. We have
discussed one possible choice, but we do not claim that the above estimator
is the best one.

\subsection{Central limit theorem for the weight sequence}
\label{sec:asymptotic_variance_weight}

In this section, we state a Central Limit Theorem on the error $\t_n -
\t_\star$
.  We show that the rate of
convergence depends upon the step-size sequence $\{\gamma_n, n\geq 1 \}$ and
discuss an averaging strategy in order to reach the optimal rate of
convergence. An additional assumption is required on the sequence  $\{\gamma_n, n\geq 1 \}$: \debutA
\item \label{hyp:CLT} $\lim_n \gamma_n \sqrt{n} = 0$, and one of the following
  condition holds:
  \begin{enumerate}[\quad (i)]
  \item \label{hyp:CLT:item1} $\log(\gamma_n / \gamma_{n+1}) = \mathrm{o}(\gamma_n)$;
  \item \label{hyp:CLT:item2} $\log(\gamma_n / \gamma_{n+1}) \sim \gamma_n / \gamma_\star$
    with $\gamma_\star > d/2$.
  \end{enumerate}
\finA 
The latter conditions are satisfied
for sequences $\gamma_n = \gamma_\star /n^\alpha$, when $\alpha \in (1/2, 1)$ for~(i), or when $\alpha=1$ and $\gamma_\star > d/2$ for~(ii).
Under this additional assumption, the following result holds (see Section~\ref{proof:theo:CLT:theta} for the
proof).  

\begin{theo}
  \label{theo:CLT:theta}
  Assume that A\ref{hyp:targetpi}, A\ref{hyp:kernel}, A\ref{hyp:stepsize} and
  A\ref{hyp:CLT} hold.  Then $\{ \gamma_n^{-1/2} \left( \t_n - \t_\star
  \right), \, n\geq 1 \}$ converges in distribution to a centered Gaussian
  distribution with variance-covariance matrix $\sigma^2 U_\star$ where
  $\sigma^2 = d/2$ in case A\ref{hyp:CLT}\eqref{hyp:CLT:item1} and $\sigma^2 = \gamma_\star d /(2 \gamma_\star
  -d)$ in case A\ref{hyp:CLT}\eqref{hyp:CLT:item2},
  \begin{equation}
    \label{eq:covar_theta}
    U_\star \eqdef \int_\Xset \left\{ \hatH_{\t_\star}(x) \hatH_{\t_\star}^T(x) -
    P_{\t_\star} \hatH_{\t_\star}(x)  \  P_{\t_\star} \hatH_{\t_\star}^T(x) \right\} \, 
    \pi_{\t_\star}(x) \, \lambda(dx)\eqsp,
  \end{equation}
  and 
  \[
  \hatH_{\t_\star}\eqdef \sum_{n \geq 0} P_{\t_\star}^n \left(I -
    \pi_{\t_\star} \right) H(\cdot, \t_\star) = \sum_{n \geq 0} P_{\t_\star}^n
  \left( H(\cdot, \t_\star) - h(\t_\star) \right)\eqsp.
  \]
\end{theo}
Notice that $\hatH_{\t_\star}$ is the Poisson solution associated
to the pair~$(P_{\t_\star}, H(\cdot, \t_\star))$, namely $\hatH_{\t_\star}$ is
a solution to: find $g: \Xset \to \Rset$ such that
\[
g - P_{\t_\star} g = H(\cdot, \t_\star) - \int_\Xset H(x, \t_\star) \, \pi_{\t_\star}(x) \, 
\lambda(dx) \eqsp. 
\]
By Proposition~\ref{prop:unifergo} and the results of \cite[Chapter~17]{meyn:tweedie:2009}, 
such a function exists and is unique up to an additive constant.

Theorem~\ref{theo:CLT:theta} shows that the rate of convergence depends upon the
step-size sequence $\{\gamma_n, n\geq 1 \}$: when $\gamma_n =
\gamma_\star / n^\alpha$ for $\alpha \in (1/2, 1]$, the maximal rate
of convergence is reached with $\alpha=1$ and the rate is
$\mathrm{O}(n^{-1/2})$.  

When $\gamma_n = \gamma_\star/n$, Theorem~\ref{theo:CLT:theta} states
that $\{ \sqrt{n} \left( \t_n - \t_\star
  \right), \, n\geq 1 \}$ converges in distribution to a centered Gaussian
  distribution with variance-covariance matrix $d  U_\star \gamma_\star^2 /(2 \gamma_\star
  -d)$, which is minimum for $\gamma_\star=d$ (in which case the
  variance-covariance matrix is $d^2  U_\star$). It is actually not
  possible to further reduce the asymptotic variance by introducing a gain matrix $\Gamma$ in the
  algorithm~(\ref{eq:SSA_formulation}) which yields the update
\[
\check \t_{n+1} = \check \t_{n} + \gamma_{n+1} \Gamma \ H(X_{n+1}, \check \t_n)
\eqsp.
\]
It is proved in~\cite[Proposition 4 p.112]{benveniste:metivier:priouret:1987}
that for a large family of gain matrices (so-called ``admissible gains'') a Central
Limit Theorem still holds for the sequence of random variables $\{\sqrt{n}(\t_n
- \t_\star),n \geq 0\}$ and the minimal variance-covariance matrix $d^2
U_\star$  is indeed
reached for $\Gamma = d \gamma_\star^{-1} {\rm
  Id}$.

{F}rom a practical point of view, it is known that stochastic approximation
algorithms are more efficient when the step-size sequence decreases at a slow
rate: in the polynomial schedule, this means that $\gamma_n =
\gamma_\star/n^\alpha$ with $\alpha$ close to $1/2$. As shown by
Theorem~\ref{theo:CLT:theta}, this yields a slower rate of convergence.
Nevertheless, combining Wang-Landau update with an {\em averaging technique}
allows to reach the optimal rate of convergence and the optimal
variance-covariance matrix: by applying~\cite[Theorem 1.4]{fort:2012}, it can
be proved that $\{\sqrt{n} \left( \frac{1}{n} \sum_{k=1}^n \t_k - \t_\star
\right), n\geq 1 \}$ converges in distribution to a centered Gaussian
distribution with variance-covariance matrix $d^2 U_\star$. The proof of this
claim is along the same lines as the proof of Theorem~\ref{theo:CLT:theta} and
details are therefore omitted.

\section{Proofs}
\label{sec:proof:cvg}
In the following, we denote by $\lfloor x \rfloor$ the integer part of $x \in \mathbb{R}$ namely
the integer such that $\lfloor x \rfloor \le x < \lfloor x \rfloor +1$. We will
also use the notation $\lceil x \rceil$ for the integer such that $ \lceil x
\rceil - 1 < x \le \lceil x \rceil$.

\subsection{Proof of Proposition~\ref{prop:unifergo}}
\label{proof:prop:unifergo}
We prove (\ref{eq:minorization_bound}); the second assertion follows by
\cite[Theorem~16.2.4]{meyn:tweedie:2009}.  Since $q$ is symmetric, it holds by
definition of the Metropolis kernel that
\[
P_\t(x,A) \geq \int_A q(x,y) \left(1 \wedge \frac{\pi_\t(y)}{\pi_\t(x)}
\right) \lambda(dy) \geq \frac{\inf_{\Xset^2} q}{\sup_\Xset \pi_\t} \int_A
\pi_\t(y) \, \lambda(dy) \eqsp.
\]
Under A\ref{hyp:kernel}, $\inf_\Xset q >0$. Furthermore, since $\t(i)>0$ and
  $\t_\star(i)>0$ for any $i \in \{1, \ldots, d \}$,
\[
\sup_\Xset \pi_\t = \left(\sum_{k=1}^d \frac{\t_\star(k)}{\t(k)} \right)^{-1} \ 
\sup_\Xset \left(\sum_{i=1}^d \frac{\pi}{\t(i)} \un_{\Xset_i}\right) 
\leq \sup_\Xset
\sum_{i=1}^d \frac{\dps \frac{\pi}{\t(i)} \un_{\Xset_i}}{\dps \frac{\t_\star(i)}{\t(i)} }
\leq \frac{\sup_\Xset \pi}{\dps \min_{i \in \{1,\ldots, d \}} \t_\star(i)} \eqsp.
\]
The right-hand side is finite by A\ref{hyp:targetpi} and does not depend upon
$\t$. Therefore, \eqref{eq:minorization_bound} holds with $ \rho \eqdef
\left( \inf_{\Xset²} q \right) (\sup_\Xset \pi)^{-1} \, \min_{1\leq i\leq d}\t_\star(i)$.

\subsection{Proof of Theorem~\ref{recomp}}
\label{sec:proof_recomp}

Define the smallest index of stratum with smallest weight according to
$\theta_n$ {\em i.e.} 
\begin{equation}
  \label{eq:StratumLowestWeight}
  I_n\eqdef \min\{i \, : \, \theta_n(i)=\underline{\theta}_n\} \eqsp,
\end{equation}
where $\ut_n$ is given by (\ref{eq:def:underlinetheta}). We also introduce the
stopping times $T_k$ as the times of return in the stratum of smallest weight:
$T_0=0$ and, for $k\geq 1$, 
\[
T_k=\inf\{n>T_{k-1}\,:\,X_n\in\Xset_{I_n}\} \eqsp,
\]
with the convention that $\inf\emptyset=\infty$.  With these notations,
Theorem~\ref{recomp} is implied by the following proposition, the proof of
which is the goal of this section.
\begin{prop}
  \label{auxrecomp}
  Under A\ref{hyp:targetpi}, A\ref{hyp:kernel} and
  A\ref{hyp:stepsize}\ref{hyp:pasdecroiss}, it holds
  \begin{align}
    \label{eq:recomp_1}
   &  \PP\Big(\forall k\in\N, \ T_k<\infty \Big) = 1 \eqsp, \\
    \label{eq:recomp_2}
 &  \PP\left(\limsup_{k\to\infty} \underline{\theta}_{T_{k}-1} >0\right)=1 \eqsp,
  \end{align}
  where $\ut_n$ is given by (\ref{eq:def:underlinetheta}).
\end{prop}
When finite, the stopping times $T_k$ are such that $\underline{\theta}_{T_k} -
\ut_{T_k-1}$ admits a known increase. Indeed, by the update
rule~\eqref{eq:update_weights_linearized},
\begin{align*}
  \theta_{T_k-1}(I_{T_k})&=\frac{\theta_{T_k}(I_{T_k})}{1+\gamma_{T_k}(1-\theta_{T_k-1}(I_{T_k}))}<\theta_{T_k}(I_{T_k})\\&\leq
  \min_{j\neq I_{T_k}}\theta_{T_k}(j)<\min_{j\neq
    I_{T_k}}\frac{\theta_{T_k}(j)}{1-\g\theta_{T_k-1}(I_{T_k})}=\min_{j\neq
    I_{T_k}}\theta_{T_k-1}(j) \eqsp,
\end{align*}
so that
\begin{align}
  \label{incthetastar}
  I_{T_k-1}=I_{T_k}\mbox{ and
  }\underline{\theta}_{T_k}=\underline{\theta}_{T_k-1}(1+\gamma_{T_k}(1-\underline{\theta}_{T_k-1}))
  \eqsp.
\end{align} In the evolution from $\underline{\theta}_{T_k-1}$ to $\underline{\theta}_{T_{k+1}-1}$, the increase provided by the return to the stratum $\Xset_{I_{T_k}}$ at time $T_k$ compensates the decrease of $\underline{\theta}_n$ generated by the subsequent visits to the other strata for $n\in\{T_k+1,\ldots,T_{k+1}-1\}$, provided that $T_{k+1}-T_{k}$ is small enough. This is indeed possible since the decrease arises from multiplicative factors $1-\gamma_n \theta(I(X_n))$, where $\gamma_n \theta(I(X_n))$ is typically much smaller than the term $\gamma_{T_k}(1-\underline{\theta}_{T_k-1})$ appearing in~\eqref{incthetastar}.

\subsubsection{Proof of~\eqref{eq:recomp_1}}
To prove the first assertion, we proceed by induction on $k$ and suppose that $\PP(T_k<\infty)=1$. This assertion is true for $k=0$. To check the condition $\PP(T_{k+1}<\infty)=1$, we are going to construct a specific sequence ensuring that $X_n$ returns in the stratum of smallest weight at some point (see~\eqref{defam} below), and show that this sequence has a positive probability of occurrence (see Lemma~\ref{lem:minopam} below).

For $m\in\N$, let $\theta_{T_k+md}((1)_m) \leq \theta_{T_k+md}((2)_m)\leq
\hdots \leq \theta_{T_k+md}((d)_m)$ denote the increasing reordering of
$(\theta_{T_k+md}(i))_{1\leq i\leq d}$ (notice that
$\theta_{T_k+md}((1)_m)=\underline{\theta}_{T_k+md}$), and define
\begin{equation}
   i_m=\max\{i\leq d:\theta_{T_k+md}((i)_m)<
\underline{\theta}_{T_k+md}(1+\gamma_1)/(1-\gamma_1)\}\label{eqdefim}.
\end{equation} The indices
$(1)_m,\dots,(i_m)_m$ are all the indices of the strata with weights close
enough to the minimal weight. We then consider the sequence obtained by
visiting successively the strata with indices $(i)_m$ for $i \leq i_m$, in
decreasing order. This corresponds to the event
\begin{equation}
  \label{defam}
   A_m = \Big\{ X_{T_k+md+1}\in\Xset_{(i_m)_m},X_{T_k+md+2}\in\Xset_{(i_m-1)_m},
   \hdots,X_{T_k+md+i_m}\in\Xset_{(1)_m} \Big\} \eqsp.
\end{equation}
On $A_m$, the weights are not updated for $j\geq i_m+1$, so that
\[
\frac{\theta_{T_k+md+i_m-1}((j)_m)}{\theta_{T_k+md+i_m-1}((1)_m)}=\frac{\theta_{T_k+md}((j)_m)}{\theta_{T_k+md}((1)_m)}\geq\frac{1+\gamma_1}{1-\gamma_1}>\frac{1+\gamma_{T_k+md+i_m}(1-\theta_{T_k+md+i_m-1}((1)_m))}{1-\gamma_{T_k+md+i_m}\theta_{T_k+md+i_m-1}((1)_m)}
\eqsp,
\]
where we have used successively the definition of~$i_m$
and~A\ref{hyp:stepsize}\ref{hyp:pasdecroiss}.  The inequality between the left-most and right-most terms rewrites
$\theta_{T_k+md+i_m}((j)_m)>\theta_{T_k+md+i_m}((1)_m)$. Now, for
$j\in\{2,\hdots,i_m\}$, it holds on $A_m$
\begin{align*}
  \frac{\theta_{T_k+md+i_m}((j)_m)}{\theta_{T_k+md+i_m}((1)_m)}
  &=\frac{\theta_{T_k+md}((j)_m)}{\theta_{T_k+md}((1)_m)}\times\frac{\dps
    1+\frac{\gamma_{T_k+md+i_m+1-j}}{1-\gamma_{T_k+md+i_m+1-j}\theta_{T_k+md+i_m-j}((j)_m)}}{\dps
    1+\frac{\gamma_{T_k+md+i_m}}{1-\gamma_{T_k+md+i_m}\theta_{T_k+md+i_m-1}((1)_m)}}
  \eqsp.
\end{align*}
The second factor on the right-hand side is larger than~1 on $A_m$ since 
$\gamma_{T_k+md+i_m}\leq \gamma_{T_k+md+i_m+1-j}$ by A\ref{hyp:stepsize}\ref{hyp:pasdecroiss}
and, using the fact that
the stratum $\Xset_{(1)_m}$ is not visited until the last step,
\[
\begin{aligned}
  \theta_{T_k+md+i_m-1}((1)_m) & < \theta_{T_k+md+i_m-j}((1)_m)=\theta_{T_k+md+i_m-j}((j)_m)\times \frac{\theta_{T_k+md}((1)_m)}{\theta_{T_k+md}((j)_m)} \\
  & \leq \theta_{T_k+md+i_m-j}((j)_m) \eqsp.
\end{aligned}
\]
Therefore, the stratum with smallest weight at iteration $T_k+md+i_m$ is still $(1)_m$, which means that $I_{T_k+md+i_m}=(1)_m$ on $A_m$ and
\begin{equation}\label{majodeltatk}
  \mathrm{on} \ A_m, \qquad T_{k+1} \leq T_k+md+i_m \leq T_k+(m+1)d \eqsp.
\end{equation}
To deduce that $\PP(T_{k+1}<\infty)=1$, we use the following lemma (whose proof is postponed to Section~\ref{sec:stability_lemma_proofs}). 

\begin{lemma}
  \label{lem:minopam}
  Under A\ref{hyp:targetpi} and A\ref{hyp:kernel}, there exists a constant $p
  \in (0,1]$ not depending on $k$ such that almost-surely
   \begin{equation*}
     \forall m\in\N, \qquad \PP(A_m|{\mathcal F}_{T_k+md}) \geq p \eqsp.
   \end{equation*}
\end{lemma}

Lemma~\ref{lem:minopam} implies that, for $m\in\N$,
\begin{align*}
  \PP\big(T_{k+1}>T_k+(m+1)d\big)& \leq \PP(A_0^c\cap A_1^c\cap \hdots \cap A_m^c) \\
  & =\PE\left(\un_{\{A_0^c\cap A_1^c\cap \hdots \cap A_{m-1}^c\}}\big(1-\PP(A_m|{\mathcal F}_{T_k+md})\big)\right)\\
  &\leq (1-p)\PP(A_0^c\cap A_1^c\cap \hdots \cap A_{m-1}^c) \eqsp,
\end{align*}
which inductively leads to $\PP(T_{k+1}>T_k+(m+1)d)\leq (1-p)^{m+1}$. The conclusion follows by taking the limit $m\to\infty$ in the latter inequality.

\subsubsection{Proof of~\eqref{eq:recomp_2}}
The proof of the second assertion relies on the following lemma (proved in
Section~\ref{sec:stability_lemma_proofs}).  For $k \geq 1$, set
\begin{equation}
  \label{eq:definitions:Gm}
{\mathcal G}_k \eqdef {\mathcal F}_{T_k} \eqsp, \qquad Y_k \eqdef
\underline{\theta}_{T_{k}-1} \eqsp,
\end{equation}
where $\ut_\ell$ is defined by (\ref{eq:def:underlinetheta}).

\bigskip

\begin{lemma}
\label{lem:surm}
Let $v:(0,1]\ni t\mapsto -\ln(t)\in{\mathbb R}_+$. Assume that
A\ref{hyp:targetpi}, A\ref{hyp:kernel} and A\ref{hyp:stepsize}\ref{hyp:pasdecroiss} hold. Then, there exist $\underline{k}\in\N$ and
$\bar{y}\in (0,1)$ such that almost-surely,
\[
\forall k\geq\underline{k}, \qquad Y_k\leq\bar{y} \Rightarrow
\PE(v(Y_{k+1})|{\mathcal G}_k)\leq v(Y_k) \eqsp.
\]
\end{lemma}

We then define by induction stopping times $\sigma_m$ and $\tau_m$ as follows:
$\sigma_0=0$, and for $m\geq 1$ (with the convention $\inf\emptyset=\infty$),
\[
\tau_m=\inf\{k>\sigma_{m-1} \, : \, Y_k\leq \bar{y}\}, \qquad
\sigma_m=\inf\{k>\tau_m \, : \, Y_k>\bar{y}\} \eqsp.
\]
All possible events can then be classified using the following partition of the underlying probability space:
\[
\left\{\exists m\geq
  0:\sigma_m<\infty=\tau_{m+1}\right\}\cup\left\{\forall m\geq
  1,\sigma_m<\infty\right\}\cup\left\{\exists m\geq
  1:\tau_m<\infty=\sigma_m\right\} \eqsp.
\]
On the first two sets, $Y_k>\bar{y}$ infinitely often so that
$\limsup_{k\to\infty} Y_k\geq\bar{y}$.  To deal with the last set, one remarks that for each $m\geq 1$, the process $(v(Y_{k\wedge
  \sigma_m})-v(Y_{k\wedge \tau_m}))_{k\geq
  \underline{k}}$ is a ${\mathcal G}_k$-supermartingale by Lemma \ref{lem:surm} and is not smaller than $-v(Y_{\tau_m})\geq -v(\bar{y}(1-\gamma_1))$ by positivity and monotonicity of $v$ and definition of $\tau_m$. So this process converges almost surely to a finite limit $V_m$ as $k\to\infty$. As a consequence, on $\{\exists m\geq
  1:\tau_m<\infty=\sigma_m\}$, $(Y_k)_k$ converges a.s. to
  $\sum_{m\geq 1} \un_{\{\tau_m<\infty=\sigma_m\}}Y_{\tau_m}e^{-V_m}$. In conclusion, $\PP(\limsup_{k\to\infty}
Y_k>0)=1$.

\subsubsection{Proofs of some technical results}
\label{sec:stability_lemma_proofs}
We now provide the proofs of the previously quoted lemmas.

\begin{proof}[Proof of Lemma \ref{lem:minopam}]
By A\ref{hyp:targetpi} and A\ref{hyp:kernel}, the constant $c\stackrel{\rm def}{=}\frac{\inf_{\Xset^2}q}{\sup_{\Xset}\pi}$ is positive. The main ingredient of the proof is the following lower-bound: for all $i\in\{1,\hdots,d\}$ and $x\in\Xset$,
\begin{align}
  P_\theta(x,\Xset_{i})\geq \int_{\Xset_{i}}q(x,y)\left(1\wedge \frac{\theta(I(x))\pi(y)}{\theta(i)\pi(x)}\right)\lambda(dy)\geq c \, \theta_\star(i)\left(\frac{\theta(I(x))}{\theta(i)}\wedge 1\right)\label{minotrans}.
\end{align}
For $j\in\{1,\hdots,i_m-1\}$, it holds  on $\{X_{T_k+md+1}\in\Xset_{(i_m)_m},\hdots,X_{T_k+md+j}\in\Xset_{(i_m+1-j)_m}\}$,
\[
\begin{aligned}
  & \frac{\theta_{T_k+md+j}((i_m+1-j)_m)}{\theta_{T_k+md+j}((i_m-j)_m)} \\
  &
  =\frac{\theta_{T_k+md}((i_m+1-j)_m)}{\theta_{T_k+md}((i_m-j)_m)}\times\frac{1+\gamma_{T_k+md+j}(1-\theta_{T_k+md+j-1}((i_m+1-j)_m)}{1-\gamma_{T_k+md+j}\theta_{T_k+md+j-1}((i_m+1-j)_m)}
  \eqsp.
\end{aligned}
\]
Both factors on the right-hand side are larger than~1 (the first one by definition of the ordered indices $(i)_m$), so that, by \eqref{minotrans},
\[
P_{\theta_{T_k+md+j}}\left(X_{T_k+md+j},\Xset_{(i_m-j)_m}\right) \geq
c \, \theta_\star((i_m-j)_m) \geq c \, \ut_\star \eqsp,\] where $\ut$ is defined by
(\ref{eq:def:underlinetheta}).  Using successively the strong Markov property
of the chain $(X_n,\theta_n)_n$, a backward induction on $n$, the definition of
$i_m$, together with~\eqref{minotrans},
\begin{align*}
  &\PP(A_m|{\mathcal F}_{T_k+md})\\
  &=\PE\left(\un_{\{X_{T_k+md+1}\in\Xset_{(i_m)_m},\hdots,X_{T_k+md+i_m-1}\in\Xset_{(2)_m}\}}P_{\theta_{T_k+md+i_m-1}}(X_{T_k+md+i_m-1},\Xset_{(1)_m})|{\mathcal F}_{T_k+md}\right)\\
  &\geq c \, \ut_\star\PE\left(\left. \un_{\{X_{T_k+md+1}\in\Xset_{(i_m)_m},\hdots,X_{T_k+md+i_m-2}\in\Xset_{(3)_m}\}}P_{\theta_{T_k+md+i_m-2}}(X_{T_k+md+i_m-2},\Xset_{(2)_m})\right|{\mathcal F}_{T_k+md}\right)\\
  &\geq \left(c \, \ut_\star\right)^{i_m-1}P_{\theta_{T_k +md}}(X_{T_k+md},\Xset_{(i_m)_m})\\
  &\geq \left(c \, \ut_\star\right)^{i_m-1}c\frac{1-\gamma_1}{1+\gamma_1} \,
  \ut_\star\geq\frac{1-\gamma_1}{1+\gamma_1}\left(c \, \ut_\star\right)^{d}
  \eqsp.
\end{align*}
It then suffices to define 
\begin{equation}
   \label{defp}\dps p = \frac{1-\gamma_1}{1+\gamma_1}\left(c \, \ut_\star\right)^{d}
\end{equation}
in order to conclude the proof.
\end{proof}

\begin{proof}[Proof of Lemma~\ref{lem:surm}]
Note first that, apart from the case when $X_n \in \Xset_{I_n}$, the 
other situation ensuring that $\underline{\theta}_{n+1}>\underline{\theta}_n$ is the case when
the chain visits the stratum of smallest weight, but the weight of this stratum is then increased while the weights of the other ones are decreased, in such a manner that this stratum no longer remains the one with smallest weight. In mathematical terms, $X_{n+1}\in\Xset_{I_n}$ and
\[
\frac{\underline{\theta}_n}{1-\gamma_{n+1}\underline{\theta}_n}<\min_{j\neq I_n}\theta_n(j)\leq \frac{\underline{\theta}_n(1+\gamma_{n+1}(1-\underline{\theta}_n))}{1-\gamma_{n+1}\underline{\theta}_n},
\]
where the first inequality actually implies that $\underline{\theta}_{n+1}>\underline{\theta}_n$ and the second one that $X_{n+1}\notin\Xset_{I_{n+1}}$.

Define $\mu(\omega)=\inf\{m\geq 1:\omega\in A_{m-1}\}$. Then, recalling that $Y_k\stackrel{\rm def}{=}\underline{\theta}_{T_{k}-1}$,
\begin{align}
  \label{minoy}
   Y_{k+1}&\geq Y_k(1+\g(1-Y_k))\prod_{n=T_k}^{T_{k+1}-2}(1-\gamma_{n+1}\theta_{n}(I(X_{n+1})))\notag\\
&\geq Y_k(1+\g(1-Y_k))\prod_{n=T_k}^{T_{k}+\mu d-2}(1-\g\theta_{n}(I(X_{n+1}))),
\end{align}
where the first factor comes from the definition of $T_k$ (see \eqref{incthetastar}); the first inequality from the possibility that for some $n\in\{T_k,\hdots,T_{k+1}-2\}$,
$X_{n+1}\in\Xset_{I_n}$ and $I_{n+1}\neq I_n$; and the second inequality from
A\ref{hyp:stepsize}\ref{hyp:pasdecroiss} and the fact that $T_{k+1}\leq T_k+\mu d$
by~\eqref{majodeltatk}. By Lemma \ref{lem:minopam} $\mu$ is smaller than some geometric random variable with
parameter bounded from below by the positive constant $p$. Therefore, the number of
terms smaller than~1 in the product on the right-hand side of~\eqref{minoy} is
small. If $Y_k$ is chosen small enough, then $\g(1-Y_k)$ is much larger than
$\g Y_k$ and if $\theta_{n}(I(X_{n+1}))$ remains of the same order as
$\theta_{T_k-1}(I(X_{T_k}))=Y_k$ for $n\in\{T_k,\hdots,T_{k}+\mu d-2\}$, then,
in average, $Y_{k+1}$ will be larger than $Y_k$. Unfortunately, since for
$\t\in\Theta$, $\max_{1\leq i\leq d}\theta(i)\geq\frac{1}{d}$,
$\theta_{T_k}(i)$ is large for $i$ in some subset of $\{1,\hdots,d\}$ and we
need to control the probability for $(X_n)_{T_k+1\leq n\leq T_{k}+\mu d-1}$ to
visit the corresponding strata. Under A\ref{hyp:targetpi}, $C\stackrel{\rm
  def}{=}\frac{\sup_{\Xset}\pi}{\inf_{\Xset}\pi}<\infty$ and, for all $i\neq
j\in\{1,\hdots,d\}$ and $x\in\Xset_i$,
\begin{align}
  P_\theta(x,\Xset_{j})=\int_{\Xset_{j}}q(x,y)\left(1\wedge
    \frac{\theta(i)\pi(y)}{\theta(j)\pi(x)}\right)\lambda(dy)\leq
  C\left(\frac{\theta(i)}{\theta(j)}\wedge 1\right) \eqsp.\label{majotrans}\end{align} This ensures
that the conditional probability to choose $X_{n+1}\in\Xset_j$ with large
weight $\theta_n(j)$, given $X_n\in\Xset_i$ with low weight $\theta_n(i)$, is
small.

To quantify this intuition, define $\dps i_{T_k}\in{\rm argmax}_{1\leq i\leq d-1}\frac{\theta_{T_k}((i+1)_0)}{\theta_{T_k}((i)_0)}$. Since 
\begin{equation}\label{eqminpro}
   \prod_{i=1}^{d-1}\frac{\theta_{T_k}((i+1)_0)}{\theta_{T_k}((i)_0)}=\frac{\max_{i}\theta_{T_k}(i)}{Y_k(1+\gamma_{T_k}(1-Y_k))}\geq
\frac{1}{2dY_k} \eqsp,
\end{equation}
it holds
\[
\frac{\theta_{T_k}((i_{T_k}+1)_0)}{\theta_{T_k}((i_{T_k})_0)} \geq (2dY_k)^{-1/(d-1)}\eqsp,
\]
so that, for all $i\in\{1,\hdots, i_{T_k}\}$,
\[\theta_{T_k}((i)_0) \leq \theta_{T_k}((i_{T_k})_0) = \frac{\theta_{T_k}((i_{T_k})_0)}{\theta_{T_k}((i_{T_k}+1)_0)}\times\theta_{T_k}((i_{T_k}+1)_0)\leq (2dY_k)^{1/(d-1)} \eqsp.
\]
Hereafter, the set
\begin{equation}
  \label{eq:definition:X(k)}
  \Xset(k)\stackrel{\rm def}{=}\cup_{i\geq
  i_{T_k}+1}\Xset_{(i)_0} \eqsp,
\end{equation}
plays the role of the union of strata with large weight according to
$\theta_{T_k}$.  Define, for $m\in\N$, 
\begin{equation}
  \label{eq:definitions:rho:Bm}
  \rho_m \eqdef 1\wedge \left((2dY_k)^{1/(d-1)}\left(\frac{1+\g}{1-\g}\right)^{(m+1)d}\right),
\qquad B_m \eqdef \bigcup_{n\in\{1,\hdots,d\}}\Big\{X_{T_k+md+n}\in\Xset(k)\Big\} \eqsp.
\end{equation}
Then,
\begin{equation}
  \label{gapetmaj}
  \forall m\in\N, \ \forall n\leq (m+1)d, \qquad \frac{\max_{j\leq i_{T_k}}\theta_{T_k+n}((j)_0)}{\min_{i\geq i_{T_k}+1}\theta_{T_k+n}((i)_0)}\wedge 1\leq \rho_m  \eqsp,
\end{equation}
which implies
\begin{equation}
\forall j\in\{1,\hdots, i_{T_k}\}, n \leq (m+1)d, \qquad \theta_{T_k+n}((j)_0)\leq \rho_m \eqsp.
\label{gapetmaj2}
\end{equation}
Using \eqref{minoy}, the definition of $\mu$, then the inequality $-\ln(x)\leq \frac{1}{x}-1$, it follows
\begin{align}
  &\PE(v(Y_{k+1})|{\mathcal G}_k)-v(Y_k)+\ln(1+\g(1-Y_k))\notag\\
  &\leq -\sum_{m\in\N}\PE\left(\ln\left(\prod_{n=T_k}
      ^{T_k+(m+1)d-1}(1-\g \theta_n(I(X_{n+1})))\right)\un_{\{A_0^c\cap \hdots\cap A_{m-1}^c\cap A_m\}}|{\mathcal G}_k\right)\notag\\
  &\leq -\sum_{m\in\N}\PE\left(\ln\left(\prod_{n=T_k}
      ^{T_k+(m+1)d-1}(1-\g \theta_n(I(X_{n+1})))\right) \un_{\{A_0^c\cap \hdots\cap A_{m-1}^c\}}|{\mathcal G}_k\right)\notag\\
  &\leq \sum_{m\in\N}\sum_{l=0}^{m}E_{ml} \eqsp,\label{surm1}
\end{align}
where the numbers $E_{ml}$ are defined by decomposing the possible events using the partition
$B_0,B_0^c \cap B_1, \dots, B_0^c\cap \dots \cap B_{m-2}^c \cap B_{m-1}, 
B_0^c\cap \dots \cap B_{m-1}^c$:
\begin{equation*}
  E_{ml}= 
  \left\{ \begin{aligned}
    &\PE\bigg(\bigg((1-\g)^{-(m+1)d}-1\bigg) \un_{\{A_0^c\cap \hdots\cap A_{m-1}^c\cap B_0\}}|{\mathcal G}_{k}\bigg) & \mbox{ for $l=0$},\\
    &\PE\bigg(\bigg((1-\g)^{-(m+1)d}-1\bigg) \un_{\{A_0^c\cap \hdots\cap A_{m-1}^c\cap B_0^c\cap\hdots \cap B_{l-1}^c\cap B_l\}}|{\mathcal G}_{k}\bigg) & \mbox{ for $0<l<m$},\\
    &\PE\bigg(\bigg((1-\g\rho_m)^{-(m+1)d}-1\bigg) \un_{\{A_0^c\cap \hdots\cap A_{m-1}^c\cap B_0^c\cap\hdots \cap B_{m-1}^c\}}|{\mathcal G}_{k}\bigg) & \mbox{ for $l=m$}.
   \end{aligned}\right.
\end{equation*}
The inequality \eqref{gapetmaj2} was used for the case $l=m$. 

Let 
\[
\overline{m}\stackrel{\rm def}{=}\left(-\frac{\ln(4d^2Y_k)}{2d(d-1)\ln\left(\frac{1+\g}{1-\g}\right)}\right)^+-1.
\]
Note that
$\overline{m}$ is chosen so that $\rho_m\leq Y_k^{1/2(d-1)}$ for $0\leq m\leq
\overline{m}$. Besides, we may assume that $k$ is large enough so that
$1-p<(1-\g)^d$, $p$ being defined in Lemma~\ref{lem:minopam} (indeed, $T_k\geq
k$, and $\lim_{n\to\infty}\gamma_n=0$ by
A\ref{hyp:stepsize}\ref{hyp:pasdecroiss}).  Therefore, using Lemma
\ref{lem:minopam} for the second inequality,
\begin{align}
\sum_{m\in\N}E_{mm}\leq& \sum_{m\leq \overline{m}}\left((1-\g Y_k^{1/2(d-1)})^{-(m+1)d}-1\right)\PP(A_0^c\cap \hdots\cap A_{m-1}^c|{\mathcal G}_{k})\notag\\&+\sum_{m>\overline{m}}(1-\g)^{-(m+1)d}\PP(A_0^c\cap \hdots\cap A_{m-1}^c|{\mathcal G}_{k})\notag\\
\leq& \sum_{m\in\N}((1-\g Y_k^{1/2(d-1)})^{-(m+1)d}-1)(1-p)^m+\sum_{m>\overline{m}}(1-\g)^{-(m+1)d}(1-p)^m\notag\\
=&\frac{1-(1-\g Y_k^{1/2(d-1)})^d}{p(p+(1-\g Y_k^{1/2(d-1)})^d-1)}+\frac{1}{p+(1-\g)^d-1}\left(\frac{1-p}{(1-\g)^d}\right)^{\lfloor \overline{m}\rfloor +1}.\label{surm2}
\end{align}

The terms $E_{ml}$ for $l<m$ can be dealt with using the following lemma, the proof of which is postponed to the end of this section.

\begin{lemma}\label{lem:majopab}
  Let $A_m, B_m$ and ${\mathcal G}_{k}$ be given by (\ref{defam}),
(\ref{eq:definitions:rho:Bm}) and (\ref{eq:definitions:Gm}).  For $0\leq l<m$,
one has
$$\PP(A_0^c\cap\hdots\cap A_{m-1}^c\cap B_0^c\cap\hdots\cap B_{l-1}^c\cap B_l|{\mathcal G}_{k})\leq Cd\rho_l(1-p)^{m-1} \eqsp.$$
\end{lemma}

This lemma ensures that for $l<m$, $E_{ml}\leq
Cd((1-\g)^{-(m+1)d}-1)(1-p)^{m-1}\rho_l$. By Fubini's theorem (the terms of the
sums below are non-negative) and a reasoning similar to the one used above to
estimate the sum $\sum_{m\in\N}E_{mm}$, we obtain
\begin{align}
  \frac{1}{Cd}&\sum_{m\in\N}\sum_{l=0}^{m-1}E_{ml}\leq \sum_{l\in\N}\rho_l\sum_{m\geq l+1}((1-\g)^{-(m+1)d}-1)(1-p)^{m-1}\notag\\
  =&\sum_{l\in\N}\rho_l\left(\frac{(1-\g)^{-d}((1-\g)^{-d}(1-p))^l}{p+(1-\g)^{d}-1}-\frac{(1-p)^l}{p}\right)\notag\\
  \leq&
  Y_k^{1/2(d-1)}\left(\frac{1}{(p+(1-\g)^{d}-1)^2}-\frac{1}{p^2}\right)+\frac{1}{(p+(1-\g)^{d}-1)^2}\left(\frac{1-p}{(1-\g)^{d}}\right)^{\lfloor
    \overline{m}\rfloor +1}.\label{surm3}\end{align} Since $T_k\geq k$, and
$\lim_{n\to\infty}\gamma_n=0$ by A\ref{hyp:stepsize}\ref{hyp:pasdecroiss},
there exists a deterministic constant $\underline{k}$ such that for $k\geq
\underline{k}$, $(1-\g)^{d}\geq 1-\frac{p}{2}$ and $\ln(1+\g(1-Y_k))\geq
\frac{\g}{2}(1-Y_k)$.  In view of~\eqref{surm1}-\eqref{surm2}-\eqref{surm3},
the definition of $\overline{m}$ and
$\ln\left(\frac{1+\g}{1-\g}\right)\leq\frac{2\g}{1-\g}$, there exists a finite
constant $K$ such that, for $k\geq\underline{k}$ and $Y_k\leq 1/4d^2$,
\[
\begin{aligned}
\PE\big(v(Y_{k+1})\big|{\mathcal G}_k\big)-v(Y_k)& \leq -\frac{\g}{2}(1-Y_k) \\
& +K\left(\g Y_k^{1/2(d-1)}+\exp\left(\frac{(1-\g)\ln\left(\frac{1-p}{1-p/2}\right)\ln(4d^2Y_k)}{4d(d-1)\g}\right)\right).
\end{aligned}
\]
This implies that there exists $\bar{y}\in (0,1/4d^2]$ such that,
for all $k\geq\underline{k}$, 
\[
Y_k\leq\bar{y} \Rightarrow \PE(v(Y_{k+1})|{\mathcal G}_k)\leq v(Y_k),
\]
which concludes the proof of Lemma~\ref{lem:surm}.
\end{proof}

\begin{proof}[Proof of Lemma \ref{lem:majopab}]
  Let us first consider the case when $l\geq 1$. By Lemma \ref{lem:minopam}, 
\begin{align*}
  \PP(A_0^c\cap\hdots&\cap A_{m-1}^c\cap B_0^c\cap\hdots\cap B_{l-1}^c\cap
  B_l|{\mathcal G}_{k})\\&\leq (1-p)^{m-1-l}
  \PE\left(\left. \un_{\{A_0^c\cap\hdots\cap
        A_{l-1}^c\}} \un_{\{B_{l-1}^c\}}\PP(B_l|{\mathcal
        F}_{T_k+ld})\right|{\mathcal G}_{k}\right).
\end{align*}
 To conclude, it is therefore enough to check that $ \un_{\{B_{l-1}^c\}}\PP(B_l|{\mathcal F}_{T_k+ld})\leq Cd\rho_l$. Now,
\begin{align*}
  \un_{\{B_{l-1}^c\}}\PP(B_l|{\mathcal F}_{T_k+ld})&\leq \un_{\{X_{T_k+ld}\notin \Xset(k)\}}P_{\theta_{T_k+ld}}(X_{T_k+ld},\Xset(k))\\&+\sum_{n=1}^{d-1}\PP(X_{T_k+ld+1}\notin \Xset(k),\hdots,X_{T_k+ld+n}\notin \Xset(k),X_{T_k+ld+n+1}\in \Xset(k)|{\mathcal F}_{T_k+ld})\\
  &\leq \sum_{n=0}^{d-1}\PE(\un_{\{X_{T_k+ld+n}\notin
    \Xset(k)\}}P_{\theta_{T_k+ld+n}}(X_{T_k+ld+n},\Xset(k))|{\mathcal
    F}_{T_k+ld}),
\end{align*}
where $\un_{\{X_{T_k+ld+n}\notin \Xset(k)\}}P_{\theta_{T_k+ld+n}}(X_{T_k+ld+n},\Xset(k))\leq C\rho_l$
by~\eqref{majotrans} and~\eqref{gapetmaj}.

For the case $l=0$, we use again Lemma~\ref{lem:minopam} to obtain 
\[
\PP(A_0^c\cap\hdots\cap A_{m-1}^c\cap B_0|{\mathcal G}_{k})\leq (1-p)^{m-1}\PP(B_0|{\mathcal G}_{k}).
\] 
The second factor is still bounded from above by $Cd\rho_l$ since $X_{T_k}\notin \Xset(k)$, which gives the claimed result.
\end{proof}

\subsubsection{Adaptation of the stability result to the standard Wang-Landau update \eqref{eq:NL_update}}
\label{sec:stabnl}

Proposition~\ref{auxrecomp} and therefore Theorem \ref{recomp} still hold when the update \eqref{eq:update_weights_linearized} of the weights is replaced by the standard Wang-Landau update \eqref{eq:NL_update}. Indeed, to adapt the proof of \eqref{eq:recomp_1}, it is enough to 
\begin{itemize}
   \item modify the definition \eqref{eqdefim} of $i_m$ into $i_m=\max\{i\leq d: \theta_{T_k+md}((i)_m)\leq\underline{\theta}_{T_k+md}(1+\gamma_1)\}$ to guarantee, using the simplified evolution of ratios of weights 
\[
\frac{\t_{n+1}(i)}{\t_{n+1}(j)}=\frac{\t_{n}(i)}{\t_{n}(j)}\times \frac{1+\gamma_{n+1}\un_{\Xset_i}(X_{n+1})}{1+\gamma_{n+1}\un_{\Xset_j}(X_{n+1})}
\]
under the standard Wang-Landau update, that \eqref{majodeltatk} still holds.
\item replace accordingly the factor $\frac{1-\gamma_1}{1+\gamma_1}$ by $\frac{1}{1+\gamma_1}$ in the definition \eqref{defp} of $p$ in the proof of Lemma \ref{lem:minopam}.
\end{itemize}
To adapt the proof of \eqref{eq:recomp_2}, it is enough to
\begin{itemize}
   \item replace the factor $(1+\gamma_{T_k}(1-Y_k))$ by $\frac{1+\gamma_{T_k}}{1+\gamma_{T_k}Y_k}$ in \eqref{minoy} which causes no complication in the remaining of the proof of Lemma \ref{lem:surm} since one still has $\ln\left(\frac{1+\gamma_{T_k}}{1+\gamma_{T_k}Y_k}\right)=\ln\left(1+\frac{\gamma_{T_k}(1-Y_k)}{1+\gamma_{T_k}Y_k}\right)\geq \frac{\gamma_{T_k}}{2}(1-Y_k)$ for $k$ large enough. Notice that since for $x\geq 0$, $\frac{1}{1+x}\geq 1-x$, one can keep the product in \eqref{minoy} unchanged.
\item replace \eqref{eqminpro} by $\dps \prod_{i=1}^{d-1}\frac{\theta_{T_k}((i+1)_0)}{\theta_{T_k}((i)_0)}=\frac{\max_i\theta_{T_k}(i)(1+\gamma_{T_k}Y_k)}{Y_k(1+\gamma_{T_k})}\geq\frac{1}{2dY_k}$.
\end{itemize}

\subsection{Proof of Theorem~\ref{theo:cvg}}
\label{proof:theo:cvg}

We start by proving that the function~$V$ defined in~\eqref{eq:def:Lyapunov:V} is a Lyapunov function for the mean-field~$h$ given by~\eqref{eq:mean-field}.

\begin{prop}
  \label{prop:Lyapunov}
  Under A\ref{hyp:targetpi}, 
  \begin{enumerate}[a)]
  \item \label{prop:Lyapunov:item0} $V$ is non-negative and continuously
    differentiable on $\Theta$.
  \item \label{prop:Lyapunov:item1} $h$ is continuous on $\Theta$ and given by
    \begin{equation}
      \label{eq:def:mean-field:2}
      h(\t) = \left( \sum_{i=1}^d \frac{\t_\star(i)}{\t(i)} \right)^{-1} \ \left(
  \t_\star - \t \right) \eqsp.
    \end{equation}
  \item \label{prop:Lyapunov:item2} for any $M>0$, $\{\t \in \Theta, \ V(\t) \leq
    M\}$ is a compact subset of $\Theta$.
  \item \label{prop:Lyapunov:item3} for any $\t \in \Theta$, $\pscal{\nabla
      V(\t)}{h(\t)} \leq 0$. In addition, $\{\t \in \Theta, \ \pscal{\nabla V(\t)}{h(\t)} =0
    \} = \{\t_\star \}$.
  \end{enumerate}
\end{prop}
\begin{proof}
  \textit{\eqref{prop:Lyapunov:item0}} It is trivial to check that $V$ is $C^1$
  on $\Theta$. By Jensen's inequality,
\[
V(\t) = - \sum_{i=1}^d \t_\star(i) \log\left( \frac{\t(i)}{\t_\star(i)}\right)
\geq - \log \left( \sum_{i=1}^d \t(i) \right) =0 \eqsp.
\]
\textit{\eqref{prop:Lyapunov:item1}} For any $i \in \{1, \ldots, d \}$, we have
by (\ref{eq:def:ChampsH}) and (\ref{eq:mean-field}),
\[
h_i(\t) = \int_\Xset H_i(x,\t) \, \pi_\t(x) \, \lambda(dx) = \t(i) \int_{\Xset_i}
\pi_\t(x) \, \lambda(dx) - \t(i) \sum_{k=1}^d \t(k) \ \int_{\Xset_k} \pi_\t(x) \ 
\lambda(dx) \eqsp.
\]
The property (\ref{eq:def:mean-field:2}) now follows upon noting that, by definition of
$\pi_\t$ (see (\ref{eq:def:pitheta})), 
\[
\int_{\Xset_k} \pi_\t(x) \, \lambda(dx) = \left( \sum_{i=1}^d
  \frac{\t_\star(i)}{\t(i)} \right)^{-1} \ \frac{\t_\star(k)}{\t(k)} \eqsp.
\]
\textit{\eqref{prop:Lyapunov:item2}} Set $M' \eqdef M -\sum_{i=1}^d \t_\star(i)
\log \t_\star(i)$. Observe that, by A\ref{hyp:targetpi}, $M' > M >0$. 
By definition of $V$ (see~(\ref{eq:def:Lyapunov:V})),
  \[
  \left\{V \leq M\right\} = \left\{\t \in \Theta, \ -\sum_{i=1}^d \t_\star(i)
    \log\t(i) \leq M' \right\} \subseteq \bigcap_{j=1}^d \left\{\t \in
    \Theta, \ \t(j) \geq m \right\} \eqsp.
  \]
  with $m \eqdef \exp(-M'/ \inf_k \t_\star(k))$. Therefore, for any $M >0$,
  there exists $m>0$ such that
\[
\{V \leq M \} \subset  \Big\{ \t \in \Theta, \ m \leq \inf_i \t(i) \leq \sup_i \t(i) \leq
1 \Big\} \eqsp.
\]
Since $V$ is continuous, $\{V \leq M \}$ is a compact subset of $\Theta$.
  
\textit{\eqref{prop:Lyapunov:item3}} By definition of $V$ and $h$ (see
(\ref{eq:def:Lyapunov:V}) and (\ref{eq:def:mean-field:2})), a simple computation shows that
\[
\begin{aligned}
\pscal{\nabla V(\t)}{h(\t)} & = -\left( \sum_{i=1}^d \frac{\t_\star(i)}{\t(i)}\right)^{-1}
\sum_{i=1}^d \frac{\t_\star(i)}{\t(i)} (\t_\star(i)-\t(i))\\
& = - \left( \sum_{i=1}^d \frac{\t_\star(i)}{\t(i)}\right)^{-1} \sum_{i=1}^d \frac{(\t_\star(i)
  -\t(i))^2}{\t(i)} \leq 0\eqsp,
\end{aligned}
\]
where we have used $\sum_{i=1}^d (\t_\star(i)-\t(i)) = 0$ to obtain the second equality.
It is also clear from the above expression that the scalar product is null if and only if 
$\t = \t_\star$.
\end{proof}

We now wish to prove that the increment $\gamma_{n+1} \left(H(X_{n+1},\t_n) -h(\t_n) \right)$ 
in~\eqref{eq:reformulation_using_mean_field} vanishes in an appropriate sense. To this end, we need some preliminary results and we rewrite the update of the weights as
\begin{equation}
   \frac{\t_{n+1}(i)}{\t_n(i)} = 1 + \gamma_{n+1} Y_{n+1}(i) \eqsp,\label{majpoids}
\end{equation}
where $Y_{n+1}(i) \eqdef \un_{\Xset_i}(X_{n+1}) - \t_n(I(X_{n+1}))$
satisfies $|Y_{n+1}(i)| \leq 1$. This key formula says that the
difference $\t_{n+1}(i)-\t_n(i)$ is not simply of order of the
step-size $\gamma_{n+1}$ but of order $\t_n(i)\gamma_{n+1}$ which
permits to circumvent the explosive behavior of the various estimates obtained in the next lemmas as $\min_{1\leq i\leq
  d}\t(i)$ tends to $0$.

\begin{lemma}
   \label{lem:RegularitePi} For any  $\t, \t' \in \Theta$, 
\[
\| \pi_\t d\lambda - \pi_{\t'}d\lambda \|_\tv \leq 2(d-1) \sum_{i=1}^d \left|1 -
  \frac{\t'(i)}{\t(i)}\right| \eqsp.  \]
\end{lemma}
\begin{proof}
  By definition of $\pi_\t$ (see~(\ref{eq:def:pitheta})),
\[
\pi_\t(x) = \sum_{i=1}^d \frac{[\t_\star(i)/ \t(i)]}{\sum_{j=1}^d[\t_\star(j)/
  \t(j)]} \ \frac{\pi(x)}{\t_\star(i)} \ \un_{\Xset_i}(x) \eqsp.
\]
Hence,
\[
  \|\pi_\t d\lambda - \pi_{\t'}d\lambda \|_\tv 
  \leq   \frac{\sum_{j=1}^d \sum_{i=1}^d \t_\star(i)\t_\star(j)\left|1/[\t(i)
      \t'(j)] - 1/[\t'(i) \t(j)]\right|}{\sum_{k=1}^d[\t_\star(k)/ \t(k)] \ \ 
    \sum_{l=1}^d[\t_\star(l)/ \t'(l)]} \eqsp.
\]
We denote by $N(\t, \t')$ the numerator of the expression of the right-hand side of the previous inequality. Then,
\begin{align*}
  N(\t, \t') &= \sum_{j=1}^d \sum_{i\neq j} \t_\star(i)\t_\star(j) \ 
  \frac{\left| \t'(i) \t(j) - \t(i) \t'(j) \right|}{\t(i) \t'(i) \t(j) \t'(j) } \\
  &\leq \sum_{j=1}^d \sum_{i\neq j} \t_\star(i)\t_\star(j) \ \left|\frac{ \t(j)
      -\t'(j)}{\t(i) \t(j) \t'(j) } \right|+ \sum_{j=1}^d \sum_{i\neq j}
  \t_\star(i)\t_\star(j) \ \left| \frac{ \t(i) -\t'(i) }{\t(i) \t'(i) \t(j)
    } \right|  \eqsp.
\end{align*}
For the denominator, we use the lower bound
\[
\forall i,j \in \{ 1,\dots,d \}, \qquad
\sum_{k=1}^d[\t_\star(k)/ \t(k)] \ \ \sum_{l=1}^d[\t_\star(l)/ \t'(l)] \geq
\frac{\t_\star(i) \t_\star(j)}{\t(i) \t'(j)}.
\]
Therefore,
\[
\|\pi_\t d\lambda - \pi_{\t'}d\lambda \|_\tv \leq 2\sum_{j=1}^d \sum_{i\neq j} \left|\frac{ \t(j)
    -\t'(j)}{ \t(j) } \right| \leq 2 (d-1) \sum_{j=1}^d \frac{\left|\t(j)
    -\t'(j)\right|}{\t(j)} \eqsp,
\]
which gives the claimed result.
\end{proof}

\begin{lemma}
    \label{lem:RegulariteKernel} For any  $\t, \t' \in \Theta$ and any $x\in \Xset$ such that $\pi_\t(x) \leq \pi_{\t'}(x)$,
    \[
    \| P_\t(x,\cdot) - P_{\t'}(x,\cdot) \|_\tv \leq 2 \ \left( 2 \sup_{i \in
      \{1, \ldots, d\}} \left| 1 - \frac{\t(i)}{\t'(i)}\right| + \sup_{i
      \in \{1, \ldots, d\} } \left| 1 - \frac{\t'(i)}{\t(i)}\right|\ \right)\ \eqsp.
          \]
\end{lemma}
\begin{proof}
  For any $x \in \Xset_j$ and $y \in \Xset_k$, we have by definition of
  $\pi_\t$ (see (\ref{eq:def:pitheta}))
\begin{equation}
  \label{eq:RatioPi}
  \frac{\pi_\t(x) \pi_{\t'}(y)  }{\pi_{\t}(y) \pi_{\t'}(x)} = \frac{\t(k)}{\t(j)}
\frac{\t'(j)}{\t'(k)}  \eqsp, \qquad \frac{\pi_\t(x)}{\pi_{\t'}(x)} = \frac{\t'(j)}{\t(j)} \eqsp.
\end{equation}
Since $P_\t$ is a Metropolis kernel, for any bounded measurable function $f$,
\begin{align*}
  \left| P_\t f(x) - P_{\t'} f(x) \right| & = \left| \int_\Xset q(x,y) \left(
      \alpha_\t(x,y) -
      \alpha_{\t'}(x,y) \right) \left(f(y) - f(x) \right) \lambda(dy)  \right|\eqsp, \\
  & \leq 2 \sup_\Xset |f| \ \sup_{\Xset^2} \left|\alpha_\theta - \alpha_{\t'}
  \right| \eqsp,
\end{align*}
with $\alpha_\t(x,y) = 1 \wedge (\pi_\t(y) / \pi_\t(x))$.  Let us distinguish
all the cases:

$\bullet$ $\pi_\t(y) \leq \pi_\t(x)$ and $\pi_{\t'}(y) \leq \pi_{\t'}(x)$.
Then,
\begin{align*}
  \left| \alpha_\t(x,y) - \alpha_{\t'}(x,y) \right| &= \left|
    \frac{\pi_\t(y)}{\pi_\t(x)} - \frac{\pi_{\t'}(y)}{\pi_{\t'}(x)} \right|
  \leq \frac{ \left| \pi_\t(y) - \pi_{\t'}(y) \right|}{\pi_\t(x)} +\frac{
    \left| \pi_\t(x) - \pi_{\t'}(x) \right|}{\pi_\t(x)}  \\
  & \leq \frac{ \left| \pi_\t(y) - \pi_{\t'}(y) \right|}{\pi_\t(y)} +\frac{
    \left| \pi_\t(x) -
      \pi_{\t'}(x) \right|}{\pi_\t(x)} \\
  &\leq 2 \sup_\Xset \left| 1 - \frac{\pi_{\t'}}{\pi_\t}\right| = 2 \sup_{i \in
    \{1, \ldots, d \}} \left| 1 - \frac{\t(i)}{\t'(i)}\right| \eqsp,
\end{align*}
where we used (\ref{eq:RatioPi}) in the last equality.

$\bullet$ $\pi_\t(y) \leq \pi_\t(x)$ and $\pi_{\t'}(x) \leq \pi_{\t'}(y)$.
Since $\pi_{\t}(x) \leq \pi_{\t'}(x) \leq \pi_{\t'}(y)$, it holds
\[
\left| \alpha_\t(x,y) - \alpha_{\t'}(x,y) \right| = 1-
\frac{\pi_\t(y)}{\pi_\t(x)} \leq 1 - \frac{\pi_\t(y)}{\pi_{\t'}(y)} \leq
\sup_\Xset \left| 1 - \frac{\pi_{\t}}{\pi_{\t'}}\right| = \sup_{i \in \{1,
  \ldots, d \}} \left| 1 - \frac{\t'(i)}{\t(i)}\right| \eqsp.
\]

$\bullet$ $\pi_\t(x) \leq \pi_\t(y)$ and $\pi_{\t'}(x) \leq \pi_{\t'}(y)$.
Then, $ \left| \alpha_\t(x,y) - \alpha_{\t'}(x,y) \right| =0$.

$\bullet$ $\pi_\t(x) \leq \pi_\t(y)$ and $\pi_{\t'}(y) \leq \pi_{\t'}(x)$.
Then, using again (\ref{eq:RatioPi}),
\begin{align*}
   \left| \alpha_\t(x,y) - \alpha_{\t'}(x,y) \right| &= 1 -
\frac{\pi_{\t'}(y)}{\pi_{\t'}(x)} \leq 1 -
\frac{\pi_{\t'}(y)}{\pi_{\t'}(x)}\frac{\pi_{\t}(x)}{\pi_{\t}(y)} 
=\frac{\pi_{\t'}(x)-\pi_{\t}(x)}{\pi_{\t'}(x)}+\frac{\pi_{\t}(x)}{\pi_{\t'}(x)}\frac{\pi_{\t}(y)-\pi_{\t'}(y)}{\pi_{\t}(y)}\\
&\leq \sup_{i
  \in \{1, \ldots, d\} } \left| 1 - \frac{\t'(i)}{\t(i)}
  \right|+\sup_{i
  \in \{1, \ldots, d\} } \left| 1 - \frac{\t(i)}{\t'(i)}
  \right| \eqsp.
\end{align*}
This concludes the proof.
\end{proof}

As a corollary of Lemmas~\ref{lem:RegularitePi} and ~\ref{lem:RegulariteKernel}, we obtain the following result.

\begin{coro}
  \label{coro:RegulariteKernel}
  Under A\ref{hyp:stepsize}\ref{hyp:pasdecroiss} and
  A\ref{hyp:stepsize}\ref{hyp:stepsize:item3}, 
  \begin{equation}
    \label{eq:regularite:pi}
    \| \pi_{\t_n}d\lambda - \pi_{\t_{n+1}} d\lambda\|_\tv \leq 2 d(d-1)\, \gamma_{n+1},  \forall n \geq 0 \,
  \end{equation}
  and for any $N \geq 0$ 
  \[
  \sup_{x \in \Xset} \| P_{\t_n}(x,\cdot) - P_{\t_{n+1}}(x,\cdot) \|_\tv \leq 4 \left(1 + \frac{1}{1-\sup_{n \geq N} \gamma_{n+1}} \right) \ \gamma_{n+1}, \quad  \forall n \geq N \ .
  \]
\end{coro}

\begin{proof}
The inequality (\ref{eq:regularite:pi}) immediately 
follows from Lemma~\ref{lem:RegularitePi}, \eqref{majpoids} and the
upper bound $|Y_{n+1}(i)| \leq 1$. 
In addition, by Lemma~~\ref{lem:RegulariteKernel},
  \begin{align*}
    \|P_{\t_{n+1}}(x,\cdot)-P_{\t_n}(x,\cdot) \|_\tv & \leq 4\left(
      \sup_{i}\left|1 - \frac{\t_{n+1}(i)}{\t_{n}(i)} \right| +\sup_{i}\left|1
        - \frac{\t_{n}(i)}{\t_{n+1}(i)} \right|
    \right)  \\
    & \leq 4\gamma_{n+1} \left( \sup_{i} |Y_{n+1}(i)| +
      \sup_{i}\frac{|Y_{n+1}(i)|}{|1+\gamma_{n+1}Y_{n+1}(i)|}\right)\eqsp.
\end{align*}
The proof is concluded upon noting that $|Y_{n+1}(i)| \leq 1$ and $1 -
\gamma_{n+1} \geq 1 - \sup_{n \geq N} \gamma_{n+1}$. 
\end{proof}

\begin{lemma}
  \label{lem:RegularitePoisson}
  Assume  A\ref{hyp:targetpi} and A\ref{hyp:kernel}. Then, for any $\t \in \Theta$, 
  there exists a function $\hatH_\t$
  solving the Poisson equation $\hatH_\t - P_\t \hatH_\t = H(\cdot, \t) -
  \pi_\t(H(\cdot,\t)) = H(\cdot, \t) - h(\t)$. In addition, 
  \[
  \sup_{\t \in \Theta, x \in \Xset} \left|\hatH_\t(x)\right| < \infty \eqsp,
  \]
  and there exists a constant $C>0$ such that for any $\t,\t' \in \Theta$,
  \[
  \sup_{\Xset} \left\{ 
  \left| \hatH_\t - \hatH_{\t'} \right| + \left| P_\t \hatH_\t - P_{\t'} \hatH_{\t'} \right| 
  \right \} \leq C \frac{ |\t - \t'|
  }{ \dps \inf_{i \in \{1,\ldots, d\}} \left\{ \t(i) \wedge \t'(i) \right\} }
  \eqsp.
  \]
\end{lemma}

\begin{proof}
  Since $\sup_{\t \in \Theta} \sup_{x \in \Xset} |H(x, \t)| \leq 1$, the
  results of Proposition~\ref{prop:unifergo} show that $\hatH_\t$ exists for
  any $\t \in \Theta$ and (see
  \textit{e.g.}~\cite[Section~17.4.1]{meyn:tweedie:2009})
\begin{equation}
  \label{eq:upperbound:HatH}
\sup_{\t \in \Theta} \sup_{x \in \Xset} \left|\hatH_\t(x)\right| 
\leq \sup_{\t \in \Theta} \sup_{x \in \Xset} \sum_{n \geq 0} \left| P_\t^n
H(\cdot,\t)(x) - \pi_\t(H(\cdot, \t)) \right| \leq \frac{2}{\rho} \eqsp. 
\end{equation}
In addition, in view of Proposition~\ref{prop:unifergo} and 
\cite[Lemma~4.2.]{fort:moulines:priouret:2011}, there exists a constant $C$ such
that, for any $\t, \t' \in \Theta$,
\begin{multline*}
  \sup_\Xset \left| P_{\t}\hatH_{\t} - P_{\t'}\hatH_{\t'} \right| +
  \sup_{\Xset} \left| \hatH_\t - \hatH_{\t'} \right|
\\ \leq C \left(
    \sup_\Xset \left| H(\cdot,\t) - H(\cdot,{\t'}) \right| + \sup_{x \in \Xset}
    \| P_\t(x,\cdot) - P_{\t'}(x,\cdot) \|_\tv + \|\pi_\t d\lambda - \pi_{\t'} d\lambda\|_\tv
  \right) \eqsp.
\end{multline*}
By definition of $H$ (see~(\ref{eq:def:ChampsH})), there exists a constant $C'$
such that for any $\t, \t' \in \Theta$,
\[
\sup_\Xset \left| H(\cdot,\t) - H(\cdot,{\t'}) \right| \leq C' | \t - \t'|
\eqsp.
\]
The proof is then concluded by Lemmas~\ref{lem:RegularitePi} and
\ref{lem:RegulariteKernel}.
\end{proof}

\begin{prop}
\label{prop:remainder:cvg}
Assume A\ref{hyp:targetpi}, A\ref{hyp:kernel} and
A\ref{hyp:stepsize}. 
Then, almost-surely,
\[
\limsup_{k \to \infty} \ \sup_{\ell \geq k}\left|\sum_{n=k}^\ell \gamma_{n+1}
  \left(H(X_{n+1},\t_n) -h(\t_n) \right) \right| = 0 \eqsp.
\]
\end{prop}
\begin{proof}
We decompose the increment into a martingale term and two remainders,
using the function $\hatH_\t$ defined in Lemma~\ref{lem:RegularitePoisson}:
\[
H(X_{n+1},\t_n) -h(\t_n) = \hatH_{\t_n}(X_{n+1}) -
P_{\t_n}\hatH_{\t_n}(X_{n+1}) =  M_{n+1} + R_{n+1}^{(1)} + R_{n+1}^{(2)} \eqsp,
\]
with
\begin{align*}
  M_{n+1} &= \hatH_{\t_n}(X_{n+1}) - P_{\t_n}\hatH_{\t_n}(X_{n})  \eqsp, \\
  R_{n+1}^{(1)} &= P_{\t_n}\hatH_{\t_n}(X_{n}) - P_{\t_{n+1}}\hatH_{\t_{n+1}}(X_{n+1}) \eqsp, \\
  R_{n+1}^{(2)} &= P_{\t_{n+1}}\hatH_{\t_{n+1}}(X_{n+1}) -
  P_{\t_{n}}\hatH_{\t_{n}}(X_{n+1}) \eqsp.
\end{align*}
Observe that $(M_n)_{n \geq 1}$ is a martingale-increment such that $\sum_n
\gamma_n^2 \PE\left[|M_n|^2 \right] < \infty$ by
A\ref{hyp:stepsize}\ref{hyp:stepsize:item3} and
Lemma~\ref{lem:RegularitePoisson}. Hence (see \textit{e.g.}~\cite[Corollary
2.2]{hall:heyde:1980}) $\limsup_{k} \ \sup_{\ell \geq k}\left|\sum_{n=k}^\ell
  \gamma_{n+1} M_{n+1}\right| = 0$ almost surely.

Consider now $\sum_{n=k}^\ell \gamma_{n+1} R_{n+1}^{(1)}$. 
Note that $R_{n+1}^{(1)}$ is a telescopic sum. We therefore resort to 
Abel's transform, and obtain
\[
\sum_{n=k}^\ell \gamma_{n+1} R_{n+1}^{(1)} = \gamma_{k+1} P_{\t_{k}}
\hatH_{\t_{k}}(X_{k}) - \gamma_{\ell+1} P_{\t_{\ell+1}}
\hatH_{\t_{\ell+1}}(X_{\ell+1}) + \sum_{n=k+1}^{\ell} \left( \gamma_{n+1}
  -\gamma_{n} \right) P_{\t_n} \hatH_{\t_n}(X_n) \eqsp.
\]
In view of Lemma~\ref{lem:RegularitePoisson} and
A\ref{hyp:stepsize}\ref{hyp:pasdecroiss}, there exists a constant $C$ such that
for any $\ell \geq k$,
\[
\left| \sum_{j=k}^\ell \gamma_{j+1} R_{j+1}^{(1)}\right| \leq C \left( \sup_{j \geq
    k} \gamma_{j+1} + \sum_{j = k+1}^\ell \left|\gamma_{j+1} -\gamma_{j} \right|
\right)  \leq 2 C \gamma_{k+1}\eqsp.
\]
Assumption~A\ref{hyp:stepsize}\ref{hyp:pasdecroiss} then implies
that $\limsup_{k} \ \sup_{\ell \geq k}\left|\sum_{n=k}^\ell \gamma_{n+1}
  R_{n+1}^{(1)}\right| = 0$ almost surely.

We finally turn to $\sum_{n=k}^\ell \gamma_{n+1} R_{n+1}^{(2)}$.
Lemma~\ref{lem:RegularitePoisson} combined with assumption~A\ref{hyp:kernel} imply,
after manipulations similar to the ones used in the proof of 
Corollary~\ref{coro:RegulariteKernel}, that there
exists a constant $C''$ such that for any $j \geq 0$, 
\[
\sup_{\Xset} \left|
P_{\t_{j+1}}\hatH_{\t_{j+1}} - P_{\t_{j}}\hatH_{\t_{j}} \right| \leq
C'' \gamma_{j+1} \eqsp. 
\]
Then, by assumption A\ref{hyp:stepsize}\ref{hyp:stepsize:item3}, $\sum_n
\gamma_n \left|R_n^{(2)}\right|$  exists almost-surely, which implies that
\[
\PP\left( \limsup_{k} \ \sup_{\ell \geq k}\left|\sum_{n=k}^\ell \gamma_{n+1}
  R_{n+1}^{(2)}\right| = 0 \right) = 1 \eqsp.
\]
This gives the claimed result.
\end{proof}

The proof of Theorem~\ref{theo:cvg} is now concluded by resorting
to~\cite[Theorems~2.2 and~2.3.]{andrieu:moulines:priouret:2005}.
Theorem~\ref{recomp} and Propositions~\ref{prop:Lyapunov}
and~\ref{prop:remainder:cvg} prove that the assumptions of these theorems hold.

\subsection{Proof of Theorem~\ref{theo:ergoLLN:X}}
\label{proof:theo:ergoLLN:X}

\begin{proof}[Proof of \eqref{theo:ergoLLN:X:item1}] 
The proof is based on~\cite[Theorem 2.1]{fort:moulines:priouret:2011}. 
We successively check the assumptions required to apply this result.
First, the condition~A1 of~\cite{fort:moulines:priouret:2011} holds since 
$\pi_\t P_\t = \pi_\t$ by assumption~A\ref{hyp:kernel}.

We now turn to condition A2 in~\cite{fort:moulines:priouret:2011}.
Fix $\eps >0$. By Proposition~\ref{prop:unifergo},
\[
\PE\left[ \| P_{\t_{n-r_\eps}}^{r_\eps}(X_{n-r_\eps}) -
    \pi_{\t_{n-r_\eps}} \, d\lambda \|_\tv \right] \leq \eps
\]
by choosing $r_\eps > \ln( \eps/2 ) / \ln(1-\rho)$. The constant sequence $r_\eps(n) = r_\eps$ is non-increasing
and obviously satisfies $r_\eps(n) /n \to0$.  Furthermore, by
Corollary~\ref{coro:RegulariteKernel}, there exists a constant $C$ (independent
of $\eps$) such that
\begin{align*}
  \sum_{j=1}^{r_\eps-1} & \PE\left[ \sup_{x \in \Xset} \| P_{\t_{n
        -r_\eps +j}}(x,\cdot)  - P_{\t_{n -r_\eps }}(x,\cdot)   \|_{\tv}\right]  \\
&  \ \leq \sum_{j=1}^{r_\eps-1} \sum_{\ell=0}^{j-1} \PE\left[ \sup_{x \in
      \Xset} \| P_{\t_{n -r_\eps +\ell+1}}(x,\cdot) -
    P_{\t_{n -r_\eps -\ell }}(x,\cdot) \|_{\tv}\right]
  \leq C \ \sum_{j=1}^{r_\eps-1} \sum_{\ell=0}^{j-1} \gamma_{n -r_\eps
    +\ell+1} \mathop{\longrightarrow}_{n\to\infty} 0
\end{align*}
since the last sum is composed of a finite number of terms, each of them going to~0 in view of assumption~A\ref{hyp:stepsize}\ref{hyp:pasdecroiss}
This ensures that condition A2 in~\cite{fort:moulines:priouret:2011} holds.

Finally, Theorem~\ref{theo:cvg} and Lemma~\ref{lem:RegularitePi} imply that
$\lim_n \int_\Xset f(x) \pi_{\t_n}(x) \, \lambda(dx) = \int_\Xset f(x) \pi_{\t_\star}(x) \, \lambda(dx)$ almost-surely.
\end{proof}

\begin{proof}[Proof of~\eqref{theo:ergoLLN:X:item2}] 
We check the conditions of \cite[Theorem 2.7]{fort:moulines:priouret:2011}.
First, the condition A3 of \cite{fort:moulines:priouret:2011} holds with $V=1$
(with the notation of~\cite{fort:moulines:priouret:2011}) in view of
Proposition~\ref{prop:unifergo}. Observe indeed that since $V=1$, $P_\t V(x) = 1 = c
+ (1-c)$ for any $c \in (0,1)$ thus showing the drift inequality. In addition,
by Proposition~\ref{prop:unifergo}, $P_\t(x, A) \geq \rho \int_A \pi_\t(x) \, 
\lambda(dx)$ for any $x \in \Xset, A \in \Xsigma$: this implies (i)
the minorization condition on the kernel $P_\t$; (ii) $\pi_\t \, 
d\lambda$ is an irreducible measure and $P_\t$ is psi-irreducible;
(iii) and $P_\t$ is strongly aperiodic since $\Xset$ is small for
$P_\t$ (see \cite[Section 5.4.3]{meyn:tweedie:2009}).

In addition, by Corollary~\ref{coro:RegulariteKernel}, there exists a constant
$C$ such that
\[
\sum_{k \geq 1} \frac{1}{k} \ \sup_{x \in \Xset} \|P_{\t_k}(x,\cdot) -
P_{\t_{k-1}}(x,\cdot) \|_\tv \leq C \, \sum_{k \geq 1} \frac{\gamma_k}{k} 
\leq C \sum_{k \geq 1} \left(\gamma_k^2 + \frac{1}{k^2}\right) < \infty
\]
by A\ref{hyp:stepsize}\ref{hyp:stepsize:item3}. This shows that the condition A4 of
\cite{fort:moulines:priouret:2011} holds.  Finally, the condition A5 of
\cite{fort:moulines:priouret:2011} is trivially satisfied in the case under consideration 
(since $V=1$ with the notation of~\cite{fort:moulines:priouret:2011}).
\end{proof}

\subsection{Proof of Theorem~\ref{theo:ergoLLN:Stratified}}
\label{proof:theo:ergoLLN:Stratified}

\begin{proof}[Proof of \eqref{theo:ergoLLN:Stratified:item1}]
We write
\begin{align*}
  \PE\left[ \sum_{i=1}^d \t_n(i) \ f(X_n) \ \un_{\Xset_i}(X_n) \right] &=
  \PE\left[ \sum_{i=1}^d \{ \t_n(i) - \t_\star(i) \} \ f(X_n) \ 
    \un_{\Xset_i}(X_n) \right]  \\
  &\quad + \sum_{i=1}^d \t_\star(i) \, \PE\left[ f(X_n) \ \un_{\Xset_i}(X_n) \right]
  \eqsp.
\end{align*}
Theorem~\ref{theo:cvg} and the dominated convergence theorem imply that the
first term in the right-hand side converges to zero. By Theorem~\ref{theo:ergoLLN:X}, the
second term converges to
\[
\sum_{i=1}^d \t_\star(i) \int_{\Xset_i} f \, \pi_{\t_\star} \, d \lambda 
= \frac{1}{d}\sum_{i=1}^d \t_\star(i) \int_{\Xset_i} f \frac{\pi}{\t_\star(i)}
\, d \lambda =\frac{1}{d} \int_\Xset f \, \pi \, d \lambda \eqsp,
\]
which gives the claimed result.
\end{proof}

\begin{proof}[Proof of \eqref{theo:ergoLLN:Stratified:item2}]
We write
\[
\frac{1}{d} \, \mathcal{I}_n(f) = \frac1n \sum_{i=1}^d \left( \t_n(i) - \t_\star(i)
\right) \sum_{k=1}^nf(X_k) \un_{\Xset_i}(X_k) + \sum_{i=1}^d \t_\star(i) 
\left[ \frac1n \sum_{k=1}^nf(X_k) \un_{\Xset_i}(X_k)\right]
\eqsp.
\]
We have
\[
\frac{1}{n}\left| \sum_{i=1}^d \ \left( \t_n(i) - \t_\star(i) \right) \ \sum_{k=1}^nf(X_k)
  \un_{\Xset_i}(X_k) \right| \leq \sup_\Xset |f| \ \sum_{i=1}^d \ \left|
  \t_n(i) - \t_\star(i) \right|
\]
and the right-hand side converges to zero almost-surely by Theorem~\ref{theo:cvg}. In
addition, by Theorem~\ref{theo:ergoLLN:X},
\[
\frac{1}{n}\sum_{k=1}^nf(X_k) \un_{\Xset_i}(X_k) \aslim\int_{\Xset_i} f \, \pi_{\t_\star} \, d\lambda= \frac{1}{d \ \t_\star(i)} \int_{\Xset_i} f \, \pi \, d\lambda \eqsp.
\]
This concludes the proof.
\end{proof}

\subsection{Proof of Theorem~\ref{theo:CLT:theta}}
\label{proof:theo:CLT:theta}
We write $ H(X_{n+1},\t_n) = h(\t_n) + e_{n+1} + r_{n+1}$ with
\[
e_{n+1} \eqdef  \hatH_{\t_n}(X_{n+1}) - P_{\t_n} \hatH_{\t_n}(X_n) \eqsp, \qquad 
r_{n+1} \eqdef  P_{\t_n}\hatH_{\t_n}(X_{n})  -
P_{\t_{n}}\hatH_{\t_{n}}(X_{n+1})\eqsp.
\]
The result follows from \cite[Theorem 1.1]{fort:2012}. We check below the
various conditions necessary to apply this theorem,
and finally establish the expression of the limiting variance. Notice
that, in our context, the event $\{ \lim_q \t_q = \t_\star \}$
has probability $1$ (by Theorem~\ref{theo:cvg}), and thus, the multiplicative factor $\un_{\{
  \lim_q \t_q = \t_\star \}}$ which appears in all the conditions
in~\cite[Theorem 1.1]{fort:2012} can be omitted below.

\begin{proof}[Condition C1] 
  The vector $\t_\star$ is a zero of the mean field $h$ in view
  of~(\ref{eq:def:mean-field:2}), and $h$ is twice continuously differentiable
  in a neighborhood of $\t_\star$ under A\ref{hyp:targetpi}. From
  (\ref{eq:def:mean-field:2}), it is easily checked that $\nabla h(\t_\star) =
  - d^{-1} \mathrm{I}$, so that $\nabla h(\t_\star)$ is a Hurwitz matrix. This
  gives condition C1 of \cite{fort:2012}.
\end{proof}

\begin{proof}[Condition C2] 
  By definition, $\{e_n, n\geq 0 \}$ is a martingale increment and by
  Lemma~\ref{lem:RegularitePoisson}, it is bounded, so that conditions C2a and
  C2b of \cite{fort:2012} follow with $\un_{\mathcal{A}_{m,k}} =
  \un_{\mathcal{A}_m}$ equal to the constant function $\un$.
We now consider C2c. A simple computation shows that 
$\PE\left[e_{k+1} e_{k+1}^T \vert \F_{k} \right] = \Xi(X_{k},\t_{k})$ with
\[
\Xi(x,\t) \eqdef \int_\Xset P_\t(x,dy) \ \hatH_\t(y) \hatH_\t(y)^T - \left( \int_\Xset
  P_\t(x,dy) \hatH_\t(y)\right)\left( \int_\Xset P_\t(x,dy) \hatH_\t(y)\right)^T
\eqsp.
\]
We introduce the function $\widehat \Xi_\t$ solution of the Poisson equation
\[
\widehat \Xi_\t(x) - P_\t \widehat \Xi_\t(x) = \Xi(x, \t_\star) - \int_\Xset \Xi(x, \t_\star)
\, \pi_\t(x) \, \lambda(dx) \eqsp.
\]
Since $\sup_{x}|\Xi(x,\t_\star)| \leq \sup_{\t, x} \left|\hatH_\t(x)\right|^2 <
\infty$ (see~Lemma~\ref{lem:RegularitePoisson}), by
Proposition~\ref{prop:unifergo} and \cite[Section~17.4.1]{meyn:tweedie:2009},
such a function exists and $\sup_{\t \in \Theta, x \in \Xset} \left|\widehat
  \Xi_\t(x)\right| < \infty$.  Using the previous equality with $x$ replaced by
$X_{k}$ and $\t$ replaced by $\t_{k-1}$, we obtain
\[
\begin{aligned}
  \Xi(X_{k}, \t_{k}) - & \int_\Xset \Xi(x,\t_\star) \, \pi_{\t_\star}(x) \lambda(dx) =
  \left\{ \Xi(X_{k}, \t_{k}) - \Xi(X_{k}, \t_\star)
  \right\} \\
  & \quad + \left( \int_\Xset \Xi(x,\t_\star) \, \pi_{\t_{k-1}}(x) \lambda(dx) - \int_\Xset
    \Xi(x,\t_\star) \, \pi_{\t_\star}(x) \lambda(dx) \right)  \\
  & \quad + \left( \widehat \Xi_{\t_{k-1}}(X_k) - P_{\t_{k}}\widehat \Xi_{\t_{k}}(X_{k})
  \right) + \left( P_{\t_{k}}\widehat \Xi_{\t_{k}}(X_{k}) - P_{\t_{k-1}}\widehat
    \Xi_{\t_{k-1}}(X_{k})\right).
\end{aligned}
\]
The terms on the right-hand side should be small.
This motivates therefore the following decomposition:
$\PE\left[e_{k+1} e_{k+1}^T \vert \F_{k} \right] = U_\star + D_{k}^{(1)} +
D_{k}^{(2)}$ with
\[
U_\star  \eqdef \int_\Xset  \Xi(x,\t_\star) \, \pi_{\t_\star}(x) \lambda(dx)  \eqsp, \qquad
D_{k}^{(2)} \eqdef \widehat \Xi_{\t_{k-1}}(X_k) - P_{\t_{k}}\widehat
\Xi_{\t_{k}}(X_{k}),
\]
and 
\begin{equation}
\label{eq:def_D1}
\begin{aligned}
D_k^{(1)} = \Big( \Xi(X_{k}, \t_{k}) - \Xi(X_{k}, \t_\star) \Big) & + \int_\Xset
  \Xi(x,\t_\star) \left( \pi_{\t_{k-1}}(x) - \pi_{\t_\star}(x) \right)
  \lambda(dx) \\
  & + \left( P_{\t_{k}}\widehat \Xi_{\t_{k}}(X_{k}) - P_{\t_{k-1}}\widehat
    \Xi_{\t_{k-1}}(X_{k})\right).
\end{aligned}
\end{equation}
We first prove that 
\begin{equation}
  \label{eq:proof:theo:CLT:tool2}
 \lim_{n\to\infty}  \gamma_n  \, \PE\left[ \left| \sum_{k=1}^n D_k^{(2)}\right|
 \right] = 0 \eqsp.
\end{equation}
To this end, we decompose this sum as
\begin{align*}
  \sum_{k=1}^n D_k^{(2)} & = \sum_{k=1}^n \left\{ \widehat\Xi_{\t_{k-1}}(X_k) -
    P_{\t_{k-1}}\widehat \Xi_{\t_{k-1}}(X_{k-1}) \right\} + \sum_{k=1}^n
  \left\{P_{\t_{k-1}}\widehat \Xi_{\t_{k-1}}(X_{k-1}) - P_{\t_{k}}\widehat
    \Xi_{\t_{k}}(X_{k})  \right\} \\
  &= \sum_{k=1}^n \left\{\widehat \Xi_{\t_{k-1}}(X_k) - P_{\t_{k-1}}\widehat
    \Xi_{\t_{k-1}}(X_{k-1}) \right\} + P_{\t_{0}}\widehat \Xi_{\t_{0}}(X_{0}) -
  P_{\t_{n}}\widehat \Xi_{\t_{n}}(X_{n}) \eqsp.
\end{align*}
Since $\sup_{\t \in \Theta,x \in \Xset} \left| \widehat \Xi_\t(x) \right| <
\infty$, the last two terms on the right-hand side of the above equality are
such that $\gamma_n \PE\left[\left| P_{\t_{0}}\widehat \Xi_{\t_{0}}(X_{0}) -
    P_{\t_{n}}\widehat \Xi_{\t_{n}}(X_{n}) \right|\right] \to 0$. The first
term is the sum of bounded martingale increments: by \cite[Theorem
2.10]{hall:heyde:1980}, there exists a constant $C$ such that
\begin{multline*}
  \sup_n \PE\left[\left| \sum_{k=1}^n \left\{\widehat \Xi_{\t_{k-1}}(X_k) -
        P_{\t_{k-1}}\widehat \Xi_{\t_{k-1}}(X_{k-1})  \right\}\right| \right]  \\
  \leq \sup_n \left( \PE\left[\left| \sum_{k=1}^n \left\{ \widehat\Xi_{\t_{k-1}}(X_k) -
          P_{\t_{k-1}}\widehat \Xi_{\t_{k-1}}(X_{k-1}) \right\}\right|^2 \right]
  \right)^{1/2} \leq C \sqrt{n} \eqsp.
\end{multline*}
Since $\lim_n \gamma_n \sqrt{n} = 0$, this concludes the proof of
(\ref{eq:proof:theo:CLT:tool2}). We now prove that 
\begin{equation}
  \label{eq:proof:CLT:theta:tool3}
D_k^{(1)} \aslim 0 \eqsp.
\end{equation}
We start with the first term in the definition~\eqref{eq:def_D1} of $D_k^{(1)}$.
Under A\ref{hyp:targetpi}, there exist $\eta >0$ and a random variable $N$, almost surely 
finite, such that
\begin{equation}
  \label{eq:proof:theo:CLT:tool1}
  \inf_{n \geq N} \inf_i \{ \t_n(i)
\wedge \t_\star(i) \} \geq \eta \qquad \mathrm{a.s.}
\end{equation}
By Lemma~\ref{lem:RegulariteKernel} and the property $\sup_{\t \in \Theta, x
  \in \Xset} \left| \hatH_\t(x) \right| < \infty$ (see
Lemma~\ref{lem:RegularitePoisson}), there exists a constant~$C$ such that for
any $x \in \Xset$, $\t, \t' \in \Theta$,
\begin{equation}
  \label{eq:proof:theo:CLT:tool4}
  \left| \Xi(x,\t) - \Xi(x,\t') \right| \leq C \left( \sup_{y\in\Xset}\ \left| \hatH_\t(y) -
  \hatH_{\t'}(y) \right| + \frac{\left| \t - \t'\right|}{\inf_i \t(i) \wedge \t'(i)}\right)\eqsp.
\end{equation}
By (\ref{eq:proof:theo:CLT:tool1}) and Lemma~\ref{lem:RegularitePoisson}, there
exists a random variable $Z$ almost-surely finite such that 
\[
\sup_{x \in \Xset} \left| \Xi(x,\t_k) - \Xi(x,\t_\star) \right|
\leq Z \, | \t_k - \t_\star |,  \qquad \mathrm{a.s.}
\]
and the right-hand side converges to zero almost surely.  For the second term
in the definition~\eqref{eq:def_D1} of $D_k^{(1)}$, we use
Lemma~\ref{lem:RegularitePi}, (\ref{eq:proof:theo:CLT:tool1}) and the bound
$\sup_{\t \in \Theta, x \in \Xset} | \Xi(x,\t)| < \infty$ to obtain the
existence of a random variable $Z$ almost-surely finite such that for any $k
\geq 1$,
\begin{align*}
  \left| \int_\Xset \Xi(x,\t_\star) \ \left\{ \pi_{\t_{k}}(x) - \pi_{\t_\star}(x)
    \right\} \lambda(dx)\right| & 
    \leq \sup_{\t \in \Theta, x \in \Xset} | \Xi(x,\t)| \ \| \pi_{\t_{k}}d\lambda -
  \pi_{\t_\star}d\lambda \|_\tv \\ & \leq Z \, | \t_{k} -
  \t_\star | \qquad \mathrm{a.s.}
\end{align*}
The right-hand side converges to zero almost surely.  Finally, for the third
term in the definition~\eqref{eq:def_D1} of $D_k^{(1)}$, by Lemma
\ref{lem:RegularitePoisson}, it can be proved that there exists a random
variable $Z$ almost-surely finite such that for any $k \geq 1$,
\[
\sup_{x \in \Xset} \left| P_{\t_{k}}\widehat \Xi_{\t_{k}}(x) -
  P_{\t_{k-1}}\widehat \Xi_{\t_{k-1}}(x) \right| \leq Z \, |\t_k -
\t_{k-1} | \qquad \mathrm{a.s.}
\]
and the right-hand side converges to zero almost surely. 
This concludes the proof of (\ref{eq:proof:CLT:theta:tool3}) and the proof of
the condition C2c of \cite{fort:2012}.
\end{proof}

\begin{proof}[Condition C3]
We write $r_{n+1} = r_{n+1}^{(1)} + r_{n+1}^{(2)}$ with
\begin{align*}
  r_{n+1}^{(1)} \eqdef P_{\t_{n+1}} \hatH_{\t_{n+1}}(X_{n+1}) - P_{\t_{n}}
  \hatH_{\t_{n}}(X_{n+1}) \eqsp, \quad r_{n+1}^{(2)}  \eqdef P_{\t_{n}}
  \hatH_{\t_{n}}(X_{n}) -P_{\t_{n+1}} \hatH_{\t_{n+1}}(X_{n+1}) \eqsp.
\end{align*}
By Lemma~\ref{lem:RegularitePoisson} and (\ref{eq:proof:theo:CLT:tool1}), there
exists a random variable $Z$ almost-surely finite such that
$|r_{n+1}^{(1)}| \leq Z \, |\t_{n+1} - \t_n| \leq Z
\gamma_{n+1} $ almost-surely. Moreover,
\[
\sqrt{\gamma_n} \ \PE\left[ \left|\sum_{k=1}^n r_k^{(2)} \right| \right] 
\leq \sqrt{\gamma_n} \, \PE\left[ \left|P_{\t_0}
  \hatH_{\t_0}(X_0) - P_{\t_n} \hatH_{\t_n}(X_n) \right| \right] \leq 2
\sqrt{\gamma_n} \sup_{x, \t} \left| \hatH_\t(x) \right| \eqsp,
\]
where the supremum in the right-hand side is finite by Lemma~\ref{lem:RegularitePoisson}.
  This concludes the proof of condition~C3 of \cite{fort:2012}.
\end{proof}

\begin{proof}[Condition C4] 
This condition is precisely assumptions~A\ref{hyp:stepsize}\ref{hyp:stepsize:item2}-\ref{hyp:stepsize:item3} and~A\ref{hyp:CLT}.
\end{proof}

\begin{proof}[Limiting variance]
  In case (i) of assumption~A\ref{hyp:CLT}, 
the limiting variance $\Sigma$ solves the equation
$\Sigma \nabla h(\t_\star)^{T} + \nabla h(\t_\star) \Sigma = - U_\star$. Since $\nabla
h(\t_\star) = - d^{-1} {\rm Id}$, it holds $\Sigma = (d /2)  U_\star$. In case~(ii), the limiting
variance solves the equation $\Sigma ({\rm Id} + 2 \gamma_\star\nabla
h(\t_\star)^{T}) + ({\rm Id}
+ 2 \gamma_\star \nabla h(\t_\star)) \Sigma = - 2 \gamma_\star U_\star$, so that 
$(d-2\gamma_\star)\Sigma  =- \gamma_\star d\, U_\star$.
\end{proof}


\subsection*{Acknowledgements}
We thank Prof. Eric Moulines for stimulating discussions.  This work is
supported by the French National Research Agency under the grants
ANR-09-BLAN-0216-01 (MEGAS) and ANR-08-BLAN-0218 (BigMC).

\bibliographystyle{amsplain}

\end{document}